\documentclass[matprg]{svjour3}
\smartqed  
\usepackage{graphicx}
\usepackage{amsmath,amssymb,amstext}
\usepackage{mathrsfs}
\usepackage{multirow}
\usepackage{threeparttable}
\usepackage[multiple]{footmisc}
\usepackage{algorithm}
\usepackage{algorithmic}
\usepackage{subfigure}
\usepackage{color}
\usepackage{cases}
\usepackage{subfig}
\usepackage{booktabs}
\usepackage{diagbox}

\makeatletter
\def\@thmcountersep{.}
\makeatother
\spnewtheorem{defi}{Definition}[section]{\bfseries}{\normalfont}
\spnewtheorem{alg}{Algorithm}[section]{\bfseries}{\normalfont}
\spnewtheorem{ass}{Assumption}[section]{\bfseries}{\normalfont}
\spnewtheorem{lem}{Lemma}[section]{\bfseries}{\normalfont}
\spnewtheorem{thm}{Theorem}[section]{\bfseries}{\normalfont}
\spnewtheorem{cor}{Corollary}[section]{\bfseries}{\normalfont}
\spnewtheorem{rem}{Remark}[section]{\bfseries}{\normalfont}

\numberwithin{equation}{section}

\begin{document}

\title{ Gradient Norm Regularization Second-Order Algorithms for Solving Nonconvex-Strongly Concave Minimax Problems\thanks{This work is supported by National Natural Science Foundation of China under the grant 12471294.}
}

\titlerunning{}        

\author{ Junlin Wang
\and Zi Xu
}


\institute{
Jun-Lin Wang \at
Department of Mathematics,  College of Sciences, Shanghai University, Shanghai 200444, P.R.China.\\ \email{wjl37@shu.edu.cn}
\and           Zi Xu \at
Department of Mathematics, College of Sciences, Shanghai University, Shanghai 200444, P.R.China.\\
Visiting Scholar at Center for Intelligent Computing, Great Bay Institute for Advanced Study, Dongguan, 523000, People's Republic of China\\
Corresponding author. \email{xuzi@shu.edu.cn}
}

\date{Received: date / Accepted: date}

\maketitle

\begin{abstract}
In this paper, we study second-order algorithms for solving nonconvex-strongly concave minimax problems, which have attracted much attention in recent years in many fields, especially in machine learning.
We propose a gradient norm regularized trust-region (GRTR) algorithm to solve nonconvex-strongly concave minimax problems, where the objective function of the trust-region subproblem in each iteration uses a regularized version of the Hessian matrix, and the regularization coefficient and the radius of the ball constraint are proportional to the square root of the gradient norm. The iteration complexity of the proposed GRTR algorithm to obtain an $\mathcal{O}(\epsilon,\sqrt{\epsilon})$-second-order stationary point is proved to be upper bounded by $\tilde{\mathcal{O}}(\ell^{1.5}\rho^{0.5}\mu^{-1.5}\epsilon^{-1.5})$, where $\mu$ is the strong concave coefficient, $\ell$ and $\rho$ are the Lipschitz constant of the gradient and Jacobian matrix respectively, which matches the best known iteration complexity of second-order methods for solving nonconvex-strongly concave minimax problems. 
We further propose a Levenberg-Marquardt algorithm with a gradient norm regularization coefficient and use the negative curvature direction to correct the iteration direction (LMNegCur), which does not need to solve the trust-region subproblem at each iteration. We also prove that the LMNegCur algorithm achieves an $\mathcal{O}(\epsilon,\sqrt{\epsilon})$-second-order stationary point within $\tilde{\mathcal{O}}(\ell^{1.5}\rho^{0.5}\mu^{-1.5}\epsilon^{-1.5})$ number of iterations.
The inexact variants of both algorithms can still obtain $\mathcal{O}(\epsilon,\sqrt{\epsilon})$-second-order stationary points with high probability, but only require $\tilde{\mathcal{O}}(\ell^{2.25}\rho^{0.25}\mu^{-1.75}\epsilon^{-1.75})$ Hessian-vector products and $\tilde{\mathcal{O}}(\ell^{2}\rho^{0.5}\mu^{-2}\epsilon^{-1.5})$ gradient ascent steps. Numerical results show the efficiency of both proposed algorithms. 
\end{abstract}

\section{Introduction}
In this paper, we consider  the following unconstrained minimax optimization problem:
\begin{equation}\label{P}
\min_{x\in\mathbb{R}^{n}} \max_{y\in\mathbb{R}^{m}} f(x,y),
\end{equation}
where $f(x,y):\mathbb{R}^{n} \times \mathbb{R}^{m} \to \mathbb{R} $ is a continuously differentiable function. Throughout this paper, we assume that the function $f(x, y)$ is $\mu$-strongly concave in $y$ for all $x\in\mathbb{R}^{n}$ but possibly nonconvex in $x$.
Recently, the minimax optimization problem has stimulated the research interest of many international scholars and has broad applications in machine learning and data science, such as generative adversarial networks (GANs) \cite{Goodfellow2014GenerativeAN,Arjovsky2017WassersteinGA}, adversarial learning \cite{Sinha2018CertifiableDR}, AUC maximization \cite{Hanley1982TheMA,Ying2016StochasticOA}, and robust optimization \cite{Ben-Tal2009RoubstOP,Gao2016DistributionallyRS}.

Three main types of optimization algorithms have been developed to solve the minimax problem \eqref{P}, including zeroth-order, first-order, and second-order optimization algorithms, which exploit the function value, gradient, and Hessian information of the objective function, respectively. Compared with zeroth-order and first-order optimization algorithms, second-order optimization algorithms have received more attention due to their faster convergence speed. In addition, second-order optimization algorithms can handle saddle points and local minima more effectively, thereby increasing the probability of finding the global optimal solution. In this paper, we focus on the second-order optimization algorithms for solving \eqref{P}.

In order to solve some convex-concave minimax optimization problems with special structures, various second-order algorithms have been proposed recently, such as implicit algorithms \cite{Monteiro2012IterationOA,Ostroukhov2020TensorMF,Bullins2022HigherMF,Jiang2022GeneralizedOM} and explicit algorithms \cite{Huang2022CubicRN,Huang2022AnAR,Jiang2024AdaptiveAO,Lin2022ExplicitSM,Nesterov2006bCubicRO}. In \cite{Adil2022LineSM,Lin2022PerseusAS}, it is proved that the iteration complexity of the second-order algorithm for solving convex-concave minimax optimization problems is lower bounded by $\Omega(\epsilon^{-2/3})$. In addition, in order to eliminate the algorithm's need for problem parameter information, such as the Lipschitz constant and the upper bound of the distance from the initial point to the optimal solution, some parameter-free second-order algorithms have been proposed in \cite{Jiang2024AdaptiveAO,Liu2022RegularizedNM,Wang2024ParameterFS}.

However, there are relatively few studies on second-order algorithms for solving nonconvex minimax optimization problems \eqref{P}. Existing second-order algorithms can be divided into two categories, i.e., cubic regularization type algorithms \cite{Luo2022FindingSS,Chen2023ACR} and trust-region type algorithms \cite{Yao2024TwoTR}. For the cubic regularization type algorithms, Luo et al. \cite{Luo2022FindingSS} proposed a Minimax Cubic Newton (MCN) algorithm, which can find an $\mathcal{O}(\epsilon,\sqrt{\epsilon})$-second-order stationary point with an iteration complexity of $\mathcal{O}(\ell^{1.5}\rho^{0.5}\mu^{-1.5} \epsilon^{-1.5})$.
MCN performs accelerated gradient ascent method to update $y$, which is then used to estimate the gradient and Hessian of $P(x) =\max_{y\in \mathbb{R}^m}f(x,y)$ involved in the cubic regularization update for $x$.
The proposed inexact variant of the MCN algorithm (IMCN) solves the cubic subproblem via perturbative gradient descent and matrix Chebyshev expansion, which requires a total of $\tilde{\mathcal{O}}(\ell^{2.5}\mu^{-1.5}\epsilon^{-2})$ Hessian-vector product calls and $\tilde{\mathcal{O}}(\ell^2\rho^{0.5}\mu^{-2}\epsilon^{-1.5})$ gradient ascent step calls to find an $\mathcal{O}(\epsilon,\sqrt{\epsilon})$-second-order stationary point with high probability.
Chen et al. \cite{Chen2023ACR} proposed an inexact Cubic-LocalMinimax (ICLM) algorithm with the same complexity bound as \cite{Luo2022FindingSS}.
 In addition to the cubic regularization type algorithm, the trust region type method is also a very representative second-order method, due to its robustness and excellent numerical performance in dealing with nonconvex problems.
For the trust-region type algorithms, Yao and Xu \cite{Yao2024TwoTR} first proposed a minimax trust-region (MINIMAX-TR) algorithm, which can be viewed as an inexact trust-region method with fixed trust-region radius of $\mathcal{O}(\sqrt{\epsilon})$, and has the same iteration complexity results as \cite{Luo2022FindingSS,Chen2023ACR}. Nevertheless, this strategy is relatively conservative, leading to poor numerical performance. To overcome this shortcoming, they further proposed a minimax trust-region algorithm with contractions and expansions (MINIMAX-TRACE) algorithm. 
However, the iteration complexity of the MINIMAX-TRACE algorithm is slightly worse, $\mathcal{O}(\ell^{4.5}\rho^{1.5}\mu^{-4.5} \epsilon^{-1.5})$, and needs function information of $P(x)$ to calculate the ratio of actual-to-predicted reduction and additional assumption that the gradient of $ P(x)$ is upper bounded by a constant. There is also no analysis of the complexity of the Hessian-vector product that allows for inexact trust region subproblem solutions.

\subsection{Contributions}
In this paper, we focus on developing a more efficient and robust second-order algorithm to solve the nonconvex-strongly concave minimax optimization problems. We summarize our contributions as follows.

We first propose a gradient norm regularized trust-region (GRTR) algorithm to solve the nonconvex-strongly concave minimax optimization problems, where the objective function of the trust-region subproblem in each iteration adopts a regularized version of the Hessian matrix, and the regularization coefficient and the radius of the ball constraint are proportional to the root mean square of the gradient norm. It can be shown that the iteration complexity of the GRTR algorithm to obtain an $\mathcal{O}(\epsilon,\sqrt{\epsilon})$-second-order stationary point is bounded by $\tilde{\mathcal{O}}(\ell^{1.5}\rho^{0.5}\mu^{-1.5}\epsilon^{-1.5})$, where $\tilde{\mathcal{O}}$ hides only absolute constants and logarithmic factors. Ignoring the logarithmic factor, this iteration complexity is consistent with the state-of-art existing second-order algorithms \cite{Luo2022FindingSS,Chen2023ACR,Yao2024TwoTR}. In addition, compared with the MINIMAX-TR algorithm \cite{Yao2024TwoTR}, our method does not need to always fix the trust-region radius, but the trust-region radius varies according to the gradient norm, which leads to better numerical performance. Compared with the MINIMAX-TRACE algorithm \cite{Yao2024TwoTR}, we do not require additional function information and do not need to assume that the gradient has a uniform upper bound. 

To solve large-scale problems, we propose an inexact GRTR algorithm variant (IGRTR) that uses either truncated conjugate gradient or negative curvature to solve the trust region subproblem inexactly. The IGRTR algorithm requires a total of $\tilde{\mathcal{O}}(\ell^{2.25}\rho^{0.25}\mu^{-1.75}\epsilon^{-1.75})$ Hessian-vector products and $\tilde{\mathcal{O}}(\ell^2\rho^{0.5}\mu^{-2}\epsilon^{-1.5})$ gradient ascent steps to find $\mathcal{O}(\epsilon,\sqrt{\epsilon})$-second-order stationary points with high probability. Compared with the IMCN algorithm \cite{Luo2022FindingSS} and the ICLM algorithm \cite{Chen2023ACR}, the proposed IGRTR algorithm has better complexity bounds in terms of Hessian-vector product calls, which is the main computational effort when solving large-scale problems.

 We further propose a Levenberg-Marquardt algorithm with a gradient norm regularization coefficient and use the negative 
curvature direction to correct the iteration direction (LMNegCur), which does not need to solve the trust-region subproblem at each iteration.  The LMNegCur algorithm can also achieve an $\mathcal{O}(\epsilon,\sqrt{\epsilon})$-second-order stationary point within $\tilde{\mathcal{O}}(\ell^{1.5}\rho^{0.5}\mu^{-1.5}\epsilon^{-1.5})$ number of iterations.

	\begin{table}[!ht]
	\centering
\caption{Comparison of second-order algorithms for solving nonconvex-strongly concave minimax problems to $\mathcal{O}(\epsilon,\sqrt{\epsilon})$-second-order stationary points.}\label{complexity}
	\begin{threeparttable}
	\centering
\begin{tabular}{c|c|c|c} \hline
	Algorithm & Type  & $\#$ of outer iterations & $\#$ of Hessian-vector products \\ \hline
	IMCN \cite{Luo2022FindingSS} & CR\tnote{1} & $\mathcal{O}(\ell^{1.5}\rho^{0.5}\mu^{-1.5}\epsilon^{-1.5})$  & $\tilde{\mathcal{O}}(\ell^{2.5}\mu^{-1.5}\epsilon^{-2})$ \\
	ICLM \cite{Chen2023ACR} & CR  & $\mathcal{O}(\ell^{1.5}\rho^{0.5}\mu^{-1.5}\epsilon^{-1.5})$  & $\tilde{\mathcal{O}}(\ell^{2.5}\mu^{-1.5}\epsilon^{-2})$ \\
	MINIMAX-TR\cite{Yao2024TwoTR}& TR\tnote{2} &$\mathcal{O}(\ell^{1.5}\rho^{0.5}\mu^{-1.5}\epsilon^{-1.5})$ & --\tnote{4} \\
	MINIMAX-TRACE\cite{Yao2024TwoTR} & TR &$\mathcal{O}(\ell^{4.5}\rho^{1.5}\mu^{-4.5} \epsilon^{-1.5})$ & --\\
	IGRTR  & TR & $\tilde{\mathcal{O}}(\ell^{1.5}\rho^{0.5}\mu^{-1.5}\epsilon^{-1.5})$  & $\tilde{\mathcal{O}}(\ell^{2.25}\rho^{0.25}\mu^{-1.75}\epsilon^{-1.75})$ \\
	ILMNegCur  & LM\tnote{3} & $\tilde{\mathcal{O}}(\ell^{1.5}\rho^{0.5}\mu^{-1.5}\epsilon^{-1.5})$   & $\tilde{\mathcal{O}}(\ell^{2.25}\rho^{0.25}\mu^{-1.75}\epsilon^{-1.75})$ 	\\	\hline
\end{tabular}
    \begin{tablenotes}
	\footnotesize
	\item [1] ``CR'' denotes cubic regularization type algorithms.
    \item [2] ``TR'' denotes trust-region type algorithms.
    \item [3] ``LM'' denotes Levenberg-Marquardt type algorithms.
	\item [4] ``--'' indicates that the Hessian-vector product complexity of the proposed algorithm is not analyzed when the subproblems are not solved exactly.
\end{tablenotes}
	\end{threeparttable}
\end{table}

\subsection{Related Works}
For solving convex-concave minimax optimization problems, there are currently a number of research results on first-order optimization algorithms, such as the gradient extrapolation algorithm \cite{Korpelevich1976TheEM}, the mirror proximal algorithm \cite{Nemirovski2004ProxWR}, the dual extrapolation algorithm \cite{Nesterov2007DualEA}, and the optimistic gradient descent algorithm (OGDA) \cite{Popov1980AMO}, which can achieve the optimal iteration complexity among algorithms that only use first-order information to solve convex-concave minimax problems, namely $\mathcal{O}(\epsilon^{-1})$\cite{Ouyang2021LowerCB}.

There are also some existing research results on first-order algorithms for solving nonconvex-strongly concave minimax problems~\cite{Bot,Jin,Lin2019,Lin2020,Lu,Rafique}. To obtain an $\epsilon$-first-order stationary point of $P(x)$, or the stationary points of $f(x,y)$, these algorithms can achieve an iteration complexity of $\tilde{\mathcal{O}}( \kappa_y^2\epsilon ^{-2} )$, where $\kappa_y$ is the condition number of $f(x,\cdot)$. In addition, Zhang et al. \cite{Zhang2021} can improve the iteration complexity to $\tilde{\mathcal{O}}(\ell^{1.5}\mu^{-0.5}\epsilon ^{-2})$ by introducing a general acceleration framework.

For general nonconvex-concave minimax problems, there are two types of algorithms, i.e., multi-loop algorithms \cite{Kong,Lin2020,Nouiehed,Ostro,Rafique,Thek2019} and single-loop algorithms \cite{Lin2019,Lu,Xu,Zhang,Zheng2023}. The MINIMAX-PPA algorithm proposed in \cite{Lin2020} achieves the best known iteration complexity, i.e., $\tilde{\mathcal{O}}( \epsilon ^{-2.5})$, among multi-loop algorithms for solving nonconvex-concave minimax problems. For solving nonconvex-concave minimax problems, the unified single-loop alternating gradient projection (AGP) algorithm \cite{Xu} was the first single-loop algorithm to achieve an iteration complexity of $\mathcal{O}( \epsilon ^{-4})$. Later, the smoothed GDA algorithm was also proposed \cite{Zhang}, achieving the same iteration complexity. For solving the special nonconvex-linear minimax problems, Pan et al. \cite{Pan} proposed a new alternating gradient projection algorithm with an iteration complexity of $\mathcal{O}( \epsilon ^{-3})$. Very recently, Xu et al. \cite{XuJiangLiuSo} proposed a riemannian alternating descent ascent algorithmic framework for nonconvex-linear minimax problems on riemannian manifolds, achieving the same complexity.

There are also some research results on the zeroth-order method for
solving the black-box minimax problems, where we can only access function values. For the nonconvex-strongly concave setting, various algorithms have been proposed, eg., the ZO-Min-Max algorithm \cite{Liu}, the zeroth-order (stochastic) gradient descent ascent (ZO-GDA) algorithm \cite{Wang}, the zeroth-order (stochastic) gradient descent multi-step ascent (ZO-GDMSA) algorithm \cite{Wang}, and the zeroth-order variance reduced gradient descent ascent
(ZO-VRGDA) algorithm \cite{Xu20}. Among these algorithms, the ZO-GDMSA algorithm achieves a complexity bound of $\tilde{\mathcal{O}}(\kappa_y^{2}\epsilon^{-2})$ in the deterministic setting, while the ZO-VRGDA algorithm achieves a complexity bound of $\mathcal{O}(\kappa_y^{3}\epsilon^{-3})$ in the stochastic setting, where $\kappa_y$ is the condition number for $f(x,\cdot)$.
For the nonconvex-concave setting, Xu et al. \cite{xu21zeroth} proposed a zeroth-order alternating stochastic gradient projection (ZO-AGP) algorithm with aniteration complexity of $\mathcal{O}(\epsilon^{-4})$. For the nonconvex-nonconcave setting, Xu et al. \cite{xu22zeroth} proposed a zeroth-order alternating gradient descent ascent (ZO-AGDA)
algorithm and a zeroth-order variance reduced alternating gradient descent ascent (ZO-VRAGDA) algorithm under the deterministic and the stochastic setting, respectivelly. The iteration complexity of the ZO-AGDA algorithm and the ZO-VRAGDA algorithm have been proved to be $\mathcal{O}(\epsilon^{-2})$ and $\mathcal{O}(\epsilon^{-3})$, respectively.


{\bfseries Notation}.
We use $\|\cdot\|$ to denote the spectral norm of matrices and Euclidean norm of vectors. We use $I_n \in \mathbb{R}^{n\times n}$ to denote identity matrix and $\lceil\cdot\rceil$ to denote the ceiling function.
For a function $f(x,y):\mathbb{R}^m\times\mathbb{R}^n\rightarrow \mathbb{R}$, we use $\nabla_{x} f(x,y)$ (or $\nabla_{y} f(x, y)$) to denote the partial gradient of $f$ with respect to the first variable (or the second variable) at point $(x, y)$. 
Additionally, we use $\nabla f(x, y)=(\nabla_{x}f(x,y), \nabla_{y}f(x,y))$ to denote the full gradient of $f$ at point $(x, y)$. 
Similarly, we denote $\nabla_{xx}^2 f(x, y)$, $ \nabla^2_{xy} f(x, y)$, $\nabla_{yx}^2f(x,y)$ and $ \nabla_{yy}^2 f(x, y)$ as the second-order partial derivatives of $f(x,y)$ at point $(x, y)$. We use the notation $\mathcal{O} (\cdot), \Omega(\cdot)$ to hide only absolute constants which do not depend on any problem parameter, and $\tilde{\mathcal{O}} (\cdot) $ to hide only absolute constants and logarithmic factors. Given a discrete set $\mathcal{S}$, we denote its cardinality by $|\mathcal{S}|$. Denote $P(x) = \max_{y\in \mathbb{R}^m}f(x,y)$, $y^*(x)=\mathop{\arg\max}_{y\in\mathbb{R}^m} f(x, y)$, and $H(x,y):= [\nabla_{xx}^2 f - \nabla_{xy}^2 f (\nabla_{yy}^2 f)^{-1}\nabla_{yx}^2 f](x,y)$.
Finally, we define the vector $z = [x;y] \in \mathbb{R}^{m+n}$.

\section{A gradient norm regularized trust-region algorithm}
 In this section, we propose a gradient norm regularized trust-region (GRTR) algorithm for solving \eqref{P}. The proposed algorithm is inspired by the universal trust-region (UTR) method \cite{Jiang2023AUT}, which is a trust-region 
 variant for solving $\min_{x\in\mathbb{R}^{n}} P(x)$. In particular, at each iteration, the UTR method solves the following gradient-regularization trust-region subproblem:
 \begin{equation*}
 	\begin{array}{l}
 		\mathop{\min}\limits_{s}~ \nabla P(x_t) ^\top s + \frac{1}{2} s^\top (\nabla^2 P(x_t)+ \sigma_t\|\nabla P(x_t)\|^{1/2}I) s \\
 		~\text{s.t.}~~ \|s\| \leqslant r_t \|\nabla P(x_t)\|^{1/2},
 	\end{array}
 \end{equation*}
 where  $\sigma_t$ and $r_t$ are iteration-dependent  parameters,  the objective function of the trust-region subproblem uses a regularized version of the Hessian matrix, and the regularization coefficient and the radius of the ball constraint are proportional to the root mean square of the gradient norm.
 The proposed GRTR algorithm can be regarded as an inexact generalization version of UTR for
 solving \eqref{P} with $P(x) := \max_{y\in\mathbb{R}^{m}} f(x, y)$,
 which consists of the following two important algorithmic components at each iteration:
 \begin{itemize}
 	\item Gradient ascent update: for a given $x_t$, run $N_t$ iterations of Nesterov's accelerated \cite{Nesterov2018Lectures} gradient ascent step to obtain an approximated maximizer 
 	\begin{align*}
 		y_t \approx y^*(x_t) = \arg\max_{y\in \mathbb{R}^m}f(x,y).
 	\end{align*}
    \item Trust-region update: for given $(x_t,y_t)$, compute a direction $s_t$ such that it is a solution of the following  trust-region subproblem:
   \begin{equation}\label{P:tr}
   	\begin{array}{l}
        \mathop{\min}\limits_{s}~ g_t ^\top s + \frac{1}{2} s^\top (H_t+ \sigma\|g_t\|^{1/2}I) s \\
        ~\text{s.t.}~~ \|s\| \leqslant r \max\{\|g_t\|^{1/2}, \epsilon^{1/2}\},
   	\end{array}
   \end{equation}
   where $g_t=\nabla_x f(x_t,y_t)$ and  $H_t=H(x_t,y_t)$ are inexact first-order and second-order information of $P(x)$ at $x_t$, respectively; $\sigma$ and $r$ are constants, which will be specified later.
 \end{itemize} 
The detailed algorithm is formally stated in Algorithm \ref{alg:1}.
\begin{algorithm}[th]
	\caption{A gradient norm regularized trust-region (GRTR) algorithm}
	\label{alg:1}
	\begin{algorithmic}
		\STATE{\textbf{Step 1}: Input $x_0,y_{-1},\eta_y, \theta, \sigma,r, N_t,\epsilon$; Set $t=0$.}
		\STATE{\textbf{Step 2}: Update $y_t$:} 
		\STATE{\quad\textbf{(a)}: Set $k=0$, $y_k^t=\tilde{y}_k^t=y_{t-1}$.}
		\STATE{\quad\textbf{(b)}: Update $y_k^t$ and $\tilde{y}_k^t$:
		\begin{align}
		y_{k+1}^t = \tilde{y}_{k}^t + \eta_y \nabla_y f(x_t, \tilde{y}_k^t), ~
			\tilde{y}_{k+1}^t = y_{k+1}^t +  \theta(y_{k+1}^t-y_{k}^t).
		\end{align}
			}
	    \STATE{\quad\textbf{(c)}: If $k \geqslant N_t-1$, set $y_t={y}_{N_t}^t$; otherwise, set $k=k+1, $ go to Step 2(b).}
       \STATE{\textbf{Step 3}: Set  $g_t=\nabla_x f(x_t,y_t)$ and  $H_t=H(x_t,y_t)$}.
        \STATE{\textbf{Step 4}: Update $s_t$ by solving the following trust-region subproblem: 
		\begin{equation}\label{tr}
	     	\mathop{\min}\limits_{\|s\| \leqslant r  \max\{\|g_t\|^{1/2}, \epsilon^{1/2}\}}  g_t ^\top s + \frac{1}{2} s^\top (H_t+ \sigma\|g_t\|^{1/2}I) s. 
		\end{equation}}
	    \STATE{\textbf{Step 5}: Update $x_t$: $x_{t+1}=x_t +s_t$.}
		\STATE{\textbf{Step 6}: If some stationary condition or stopping criterion is satisfied, stop;
        otherwise, set $t=t+1, $ go to Step 2(a).}
	\end{algorithmic}
\end{algorithm}

Note that in the trust-region subproblem \eqref{tr}, the trust-region radius varies with the gradient norm. When $\|g_t\| \leqslant \epsilon$, we fix the trust-region radius to $\mathcal{O}(\sqrt{\epsilon})$. If we set $\sigma= 0$ and fix the trust-region radius to $\mathcal{O}(\sqrt{\epsilon})$, then the GRTR algorithm is equivalent to the MINIMAX-TR algorithm \cite{Yao2024TwoTR}. Compared with the MINIMAX-TRACE algorithm \cite{Yao2024TwoTR}, the GRTR algorithm does not need to calculate the function value of $P(x)$. Note that in the trust-region subproblem \eqref{tr}, the trust-region radius varies with the gradient norm. When $\|g_t\| \leqslant \epsilon$, we fix the trust-region radius to $\mathcal{O}(\sqrt{\epsilon})$. If we set $\sigma= 0$ and fix the trust-region radius to $\mathcal{O}(\sqrt{\epsilon})$, then the GRTR algorithm is equivalent to the MINIMAX-TR algorithm \cite{Yao2024TwoTR}. Compared with the MINIMAX-TRACE algorithm \cite{Yao2024TwoTR}, the GRTR algorithm does not need to calculate the function value of $P(x)$. 

In the next section, we will prove that for solving \eqref{P} to obtain an $\mathcal{O}(\epsilon,\sqrt{\epsilon})$-second-order stationary point of $P(x)$, the iteration complexity of the GRTR algorithm  is $\tilde{\mathcal{O}}(\ell^{1.5}\rho^{0.5}\mu^{-1.5}\epsilon^{-1.5})$.  Our complexity analysis extends the framework of monotonic function value reduction of $P(x)$ from \cite{Jiang2023AUT}, categorizing iterations into distinct cases where each iteration either achieves sufficient descent in the objective value or induces linear contraction of the gradient norm. The key distinction lies in addressing the challenge of GRTR's trust-region subproblem: the inherent inexactness from approximate gradient $g_t=\nabla_x f(x_t,y_t)$ and Hessian $H_t=H(x_t,y_t)$ introduces significant challenges absent in standard trust-region analyses. Our main contribution centers on rigorously quantifying these approximation errors through novel inner-loop convergence characterizations, coupled with systematic parameter selection strategies. This enables us to establish a noise-resistant iteration process with guaranteed monotonic descent while maintaining the overall iteration complexity - critical advancements that distinguish our analysis from conventional trust-region theory.

\subsection{Complexity analysis}

Before we prove the iteration complexity of the GRTR algorithm, we first make the following assumptions about $f(x,y)$ in this paper.

\begin{ass}\label{ass}
$f(x,y)$ satifies the following assumptions: 
	\begin{enumerate}
		\item[(\romannumeral1)] $f(\cdot,\cdot)$ is $\ell$-Lipschitz continuous, i.e., for any $(x,y), (x', y')$, it holds
		\[
		\|\nabla f(x,y) - \nabla f(x', y') \|\leqslant \ell \| (x,y) - (x', y')\|.
		\]
		\item[(\romannumeral2)] The Jacobian matrices $\nabla_{xx}^2 f(x, y)$, $\nabla_{xy}^2 f(x, y)$, $\nabla_{yx}^2 f(x, y)$ and $ \nabla_{yy}^2 f(x, y)$ are $\rho$-Lipschitz continuous.
		\item[(\romannumeral3)] $P(x)$ is bounded below by a constant $P^*$.
	\end{enumerate}
\end{ass}
Next, similar to \cite{Luo2022FindingSS,Chen2023ACR,Yao2024TwoTR} , the first-order and second-order stationary points are defined as follows.
\begin{defi}\label{approximate FSP}
	Suppose that Assumption \ref{ass} holds. We call $x$ an $\epsilon$-first-order stationary point of $P(x)$ if $\|\nabla{P(x)}\| \leqslant \epsilon$.
\end{defi}

\begin{defi}\label{approximate SSP}
		Suppose that Assumption \ref{ass} holds. We call $x$ an $\mathcal{O}(\epsilon, \sqrt{\epsilon})$-second-order stationary point of $P(x)$ if $\|\nabla P(x)\| \leqslant \xi\epsilon$ and $\nabla^2 P(x) \succeq -\theta\sqrt{\epsilon}I$, where $\xi$ and $\theta$ are positive constants independent of $\epsilon$.
\end{defi}

\begin{lem}[\cite{Chen2023ACR}]\label{lem1}
Suppose that Assumption \ref{ass} holds. $P(x)$ has
 $L_1:= (\kappa+1)\ell$-Lipschitz continuous gradient with $\kappa=\ell/\mu$. Furthermore, $y^*(x)$ is well-defined and $\kappa$-Lipschitz continuous, and $\nabla P(x) = \nabla_x f(x, y^*(x))$.
\end{lem}
Denote $H(x,y)$ $=$ $[\nabla_{xx}^2 f - \nabla_{xy}^2 f (\nabla_{yy}^2 f)^{-1} \nabla_{yx}^2 f](x,y)$.
\begin{lem}[\cite{Chen2023ACR}]\label{hessian of P}
	Suppose that Assumption \ref{ass} holds. $H(x,y)$ is a Lipschitz continuous mapping with Lipschitz constant $L_H = \rho (1 + \kappa)^2$. Additionally, the Hessian of $P(x)$ satisfies $\nabla^2 P(x)=H(x,y^*(x))$, and it is Lipschitz continuous with constant $L_2 = \rho (1+\kappa)^3$.
\end{lem}


We first  present some preliminary analysis of our method. The following lemma gives the optimality conditions for the trust-region subproblem \eqref{tr}, which can be directly obtained from Theorem 4.1 in \cite{Jorge1999NumericalO} and will not be repeated here.
\begin{lem}\label{oc}
	The direction $s_t$ is a solution of \eqref{tr} if and only if there exists a dual multiplier $\lambda_t \geqslant 0 $ such that
\begin{numcases}{}	
	\| s_t\| \leqslant r  \max\{\|g_t\|^{1/2}, \epsilon^{1/2}\} \label{oc.primal}, \\
	\lambda_t (\|s_t\|-r \max\{\|g_t\|^{1/2}, \epsilon^{1/2}\} )=0 \label{oc.com},                    \\
	(H_t + \sigma \|g_t\|^{1/2} I + \lambda_t I ) s_t = -g_t, \label{oc.fir}\\
	H_t + \sigma \|g_t\|^{1/2} I + \lambda_t I \succeq 0. \label{oc.sec}
\end{numcases}
\end{lem}

In the remainder of this paper, we use $(s_t, \lambda_t)$ to denote the primal-dual solution pair of a subproblem at iteration $t$. The following lemma shows that by running a sufficient number of gradient ascent steps, we can obtain sufficiently accurate estimators of the gradient and Hessian. 
\begin{lem}\label{gradientascent}
	Suppose that Assumption \ref{ass} holds. For any given $\epsilon_1 > 0$ and $\epsilon_2 > 0$, set
	\begin{equation}\label{N_t}
		\eta_y= \frac{1}{\ell},~\theta=\frac{\sqrt{\kappa}-1}{\sqrt{\kappa}+1}, ~\begin{cases}
			N_0 \geqslant 2\sqrt{\kappa}\ln \frac{\sqrt{\kappa+1}\|y_{-1}-y^*(x_0)\|}{A}, \\
			N_t \geqslant 2\sqrt{\kappa} \ln\frac{\sqrt{\kappa+1}(A + \kappa\|s_{t-1}\|) }{A}, t \geqslant 1,
		\end{cases}
	\end{equation}
	where $A = \min\{\epsilon_1/\ell, \epsilon_2/(2L_H)\}$, then $\|\nabla P(x_t)-g_t\| \leqslant \epsilon_1 $ and $\|\nabla^2 P(x_t)-H_t\| \leqslant \epsilon_2$.
\end{lem}

\begin{proof}
	We first use induction to show that 
	\begin{align}\label{gra:1}
		\|y_t-y^{*}(x_t)\| \leqslant A
	\end{align}
    for all $t \geqslant 0$. Note that $y_t$ is obtained by applying $N_t$ Nesterov's accelerated gradient ascent steps starting from $y_{t-1}$. By Assumption \ref{ass} and (2.2.23) in \cite{Nesterov2018Lectures}, we have
    \begin{align}\label{gra:2}
    		\|y_t-y^*(x_t)\| \leqslant&~ (\sqrt{\kappa +1})e^{-N_t/(2\sqrt{\kappa})}\|y_{t-1}-y^*(x_{t})\|.
    \end{align}
    By further combining with \eqref{N_t}, \eqref{gra:1} holds directly for $t = 0$. 
   Suppose that \eqref{gra:1} holds for $t$, by \eqref{gra:2}, Lemma \ref{lem1} and \eqref{N_t}, we then obtain that
    \begin{align*}
    	\|y_{t+1}-y^*(x_{t+1})\| \leqslant&~ (\sqrt{\kappa +1})e^{-N_{t+1}/(2\sqrt{\kappa})}\|y_{t}-y^*(x_{t+1})\| \\
    	\leqslant &~ (\sqrt{\kappa +1})e^{-N_{t+1}/(2\sqrt{\kappa})}(\|y_{t}-y^*(x_{t})\| + \|y^*(x_{t+1})-y^*(x_{t})\|)\\
    	\leqslant &~ (\sqrt{\kappa +1})e^{-N_{t+1}/(2\sqrt{\kappa})}(A + \kappa\|x_{t+1}-x_{t}\|) \leqslant A,
    \end{align*}
    which completes the proof of \eqref{gra:1}. Then, 
    by Assumption \ref{ass}, Lemma \ref{lem1}, Lemma \ref{hessian of P} and \eqref{gra:1}, we get
    \begin{align*}
    	\|\nabla P(x_t) - g_t\| &= \|\nabla_x f(x_t,y^*(x_t)) - \nabla_x f(x_t,y_t)\| \leqslant l \|y_t - y^*(x_t)\| \leqslant \epsilon_1, \\
    	\|\nabla^2 P(x_t) - H_t\| &= \|H(x_t,y^*(x_t)) - H(x_t,y_t)\| \leqslant L_H \|y_t - y^*(x_t)\| \leqslant \epsilon_2.
    \end{align*}
    The proof is then completed.
\end{proof}

From the optimality condition \eqref{oc.primal}-\eqref{oc.sec} and Lemma \ref{gradientascent}, we can prove the conclusions in the following lemma and give the key inequalities about the function and gradient in two different cases $\|g_t\| \geqslant \epsilon $ and $\|g_t\| \leqslant \epsilon$.

\begin{lem}\label{lem5}
	Suppose that Assumption \ref{ass} holds. Let $\{x_t\}$ be a sequence generated by Algorithm \ref{alg:1}.
	Set 
	\begin{equation}\label{lem5:1}
		\sigma =\frac{\sqrt{L_2}}{2}, ~r=\frac{1}{4\sqrt{L_2}}, ~\epsilon_1 = \min\left\{ \frac{1}{96},\frac{\sqrt{L_2}}{16L_1}\right\} \epsilon^{3/2}, ~\epsilon_2 =\frac{\sqrt{L_2}}{12}\epsilon^{1/2},	
	\end{equation}
	and choose $\eta_y, \theta, N_t$ in \eqref{N_t}. When $\|g_t\|\geqslant \epsilon$, if   $\lambda_t>0$, we have
	\begin{equation}\label{lem5:2}
		P(x_{t+1}) \leqslant P(x_t)  -\frac{1}{128\sqrt{L_2}}\|g_t\|^{3/2}, ~\|g_{t+1}\| \leqslant 2 + \frac{L_1}{4\sqrt{L_2}}\|g_t\|^{1/2} + \|g_t\|;\\
	\end{equation}
	Otherwise if $\lambda_t=0$, we have
	\begin{equation}\label{lem5:3}
		P(x_{t+1}) \leqslant P(x_t), ~\|g_{t+1}\| \leqslant \frac{1}{3}\|g_t\|.
	\end{equation}
\end{lem}

\begin{proof}
   By Lemma \ref{hessian of P} and  the Cauchy-Swartz inequality, we obtain
   	\begin{align}\label{lem5:4}
 		&~P(x_{t+1}) - P(x_t) \nonumber\\
   		\leqslant &~\nabla P(x_t)^\top s_t + \frac{1}{2} s_t^\top \nabla^2 P(x_t) s_t + \frac{L_2}{6} \|s_t\|^3 \nonumber\\
   		=&~g_t^\top s_t + \frac{1}{2} s_t^\top H_t s_t + (\nabla P(x_t) - g_t)^\top s_t + \frac{1}{2} s_t^\top (\nabla^2 P(x_t) - H_t) s_t + \frac{L_2}{6} \|s_t\|^3 \nonumber\\
   		\leqslant&~ g_t^\top s_t + \frac{1}{2} s_t^\top H_t s_t + \|\nabla P(x_t) - g_t\| \|s_t\| + \frac{1}{2} \|\nabla^2 P(x_t) - H_t\| \|s_t\|^2 + \frac{L_2}{6} \|s_t\|^3 \nonumber\\
   		\leqslant&~ g_t^\top s_t + \frac{1}{2} s_t^\top H_t s_t + \epsilon_1\|s_t\| + \frac{\epsilon_2}{2} \|s_t\|^2 + \frac{L_2}{6} \|s_t\|^3,
   	 \end{align}
   where the last inequality is by Lemma \ref{gradientascent}.
   By \eqref{oc.fir} and \eqref{oc.sec}, it is easy to verify that
   \begin{align}\label{lem5:5}
   	   g_t^\top s_t + \frac{1}{2} s_t^\top H_t s_t 
   	   =& -\frac{1}{2} s_t^{\top} (H_t + \sigma \|g_t\|^{1/2}I + \lambda_t I ) s_t - \frac{1}{2} s_t^{\top} ( \sigma \|g_t\|^{1/2}I + \lambda_t I ) s_t \nonumber \\
   	   \leqslant& - \frac{1}{2} s_t^{\top} ( \sigma \|g_t\|^{1/2}I + \lambda_t I ) s_t.
   \end{align}
   Combining \eqref{lem5:4} and \eqref{lem5:5}, we get
   \begin{align}\label{lem5:6}
    	P(x_{t+1}) 
    	\leqslant  P(x_t) - \frac{1}{2} s_t^{\top} ( \sigma \|g_t\|^{1/2}I + \lambda_t I ) s_t + \epsilon_1\|s_t\| + \frac{\epsilon_2}{2} \|s_t\|^2 + \frac{L_2}{6} \|s_t\|^3.
   \end{align} 
   For the case  $\|g_t\| \geqslant \epsilon$ and $\lambda_t > 0$, the complementary property  \eqref{oc.com} implies that
   \begin{equation}\label{lem5:7}
   	   \|s_t\|= r\|g_t\|^{1/2}.
   \end{equation}
    By plugging \eqref{lem5:7} and \eqref{lem5:1} into \eqref{lem5:6}, we then obtain
   \begin{align}\label{lem5:8}
   	  P(x_{t+1}) - P(x_t) 
   	  \leqslant & -\frac{1}{2}\sigma r^2\|g_t\|^{3/2} + \frac{1}{96}r\epsilon\|g_t\|^{1/2} + \frac{\sqrt{L_2}}{24}r^2\epsilon^{1/2} \|g_t\| +  \frac{L_2}{6}r^3\|g_t\|^{3/2} \nonumber \\
   	  \leqslant& -\frac{1}{128\sqrt{L_2}}\|g_t\|^{3/2},
   \end{align}
   which shows the first inequaity of \eqref{lem5:2}.
   By the Cauchy-Swartz inequality, Lemmas \ref{lem1} and \ref{gradientascent}, we get
   	\begin{align}\label{lem5:9}
   	\|g_{t+1}\| \leqslant&~ \|g_{t+1}- \nabla P(x_{t+1})\| + \| \nabla P(x_{t+1})- \nabla P(x_t)\|+ \|\nabla P(x_t)-g_t\| +\|g_t\| \nonumber\\
   	\leqslant&~ 2\epsilon_1 + L_1\|s_t\| + \|g_t\|.
   \end{align}
    Plugging \eqref{lem5:7} and \eqref{lem5:1} into \eqref{lem5:9}, the second inequaity of \eqref{lem5:2} is proved.

    Next, we consider the case $\|g_t\| \geqslant \epsilon$ and $\lambda_t = 0$. Combining \eqref{lem5:6} with $\|s_t\| \leqslant r\|g_t\|^{1/2}$ and \eqref{lem5:1}, we have 
    \begin{align}\label{lem5:10}
    	&~P(x_{t+1})-P(x_t) \nonumber \\
    	\leqslant & -\bigg(\frac{\sigma}{2}-\frac{L_2r}{6}\bigg) \|g_t\|^{1/2}\|s_t\|^2 + \min\bigg\{ \frac{1}{96},\frac{\sqrt{L_2}}{16L_1}\bigg\}\epsilon^{3/2}\|s_t\| + \frac{\sqrt{L_2}}{24}\epsilon^{1/2} \|s_t\|^2 \nonumber\\
    	\leqslant& -\left(\frac{\sqrt{L_2}}{6}-\min\bigg\{ \frac{1}{96},\frac{\sqrt{L_2}}{16L_1}\bigg\}\frac{\epsilon^{3/2}}{\|g_t\|^{1/2}\|s_t\|}\right)\|g_t\|^{1/2}\|s_t\|^2.
    \end{align}
     By \eqref{oc.fir} and $\|H_t\| \leqslant \ell(1+\kappa) = L_1$, we can easily get
    \begin{equation*}
   	    \|g_t\| \leqslant (L_1+\sigma\|g_t\|^{1/2})\|s_t\|,
    \end{equation*}
  which implies that 
   \begin{align}\label{lem5:11}
   	&~ \epsilon^{3/2}\leqslant \epsilon^{1/2}\|g_t\| \leqslant (L_1+\sigma)\|g_t\|^{1/2}\|s_t\|. 
   \end{align}
   Then, the first inequaity of \eqref{lem5:3} holds by combining \eqref{lem5:10} with \eqref{lem5:11}.
   By the Cauchy-Swartz inequality, Lemmas \ref{hessian of P} and \ref{gradientascent}, \eqref{oc.fir} and $\|s_t\| \leqslant r\|g_t\|^{1/2}$, we have 
   \begin{align}\label{lem5:12}
     	\|g_{t+1}\| 
	    \leqslant&~ \|\nabla P(x_{t+1})-g_{t+1}\| + \|\nabla P(x_{t+1}) - \nabla P(x_t) - \nabla^2 P(x_t) s_t\| \nonumber\\
     	&~ + \|\nabla P(x_t)-g_t\| + \|\nabla^2 P(x_t)-H_t\|\|s_t\| + \|g_t + H_ts_t\| \nonumber\\
	    \leqslant &~ 2\epsilon_1 + \frac{L_2}{2}\|s_t\|^2 + \epsilon_2\|s_t\| + \sigma \|g_t\|^{1/2}\|s_t\| \nonumber \\
	    \leqslant &~ \left(2\frac{\epsilon_1}{\|g_t\|} + \frac{L_2}{2}r^2 +\frac{\epsilon_2}{\|g_t\|^{1/2}}r + \sigma r\right)\|g_t\|.
    \end{align}
    The second inequaity of \eqref{lem5:3} follows by plugging \eqref{lem5:1} into \eqref{lem5:12}.
\end{proof}

\begin{lem}\label{lem6}
	Suppose that Assumption \ref{ass} holds. Let $\{x_t\}$ be a sequence generated by Algorithm \ref{alg:1} with parameter settings in \eqref{lem5:1}.
	Then for the case $\|g_t\|\leqslant \epsilon$, if  $\lambda_t \leqslant \sqrt{L_2\epsilon} $, we have already found an $\mathcal{O}(\epsilon,\sqrt{\epsilon})$-second-order stationary point $x_t$; if  $\lambda_t \geqslant \sqrt{L_2\epsilon} $, we have 
	\begin{align}\label{lem6:2}
		P(x_{t+1})-P(x_t) 
		\leqslant  - \frac{1}{128}\sqrt{L_2}\epsilon^{3/2}, ~\|g_{t+1}\| \leqslant 2 +\frac{L_1}{4\sqrt{L_2}}.
	\end{align}	
\end{lem}

\begin{proof}
	We first consider the case $\|g_t\|\leqslant \epsilon$ and $\lambda_t \leqslant \sqrt{L_2\epsilon}$. 
	By $\|g_t\| \leqslant \epsilon$, Lemma \ref{gradientascent} and \eqref{lem5:1}, we have
	\begin{align}\label{lem6:3}
		\|\nabla P(x_t)\| \leqslant \|\nabla P(x_t)-g_t\| + \|g_t\| \leqslant \epsilon_1 + \|g_t\| \leqslant \frac{97}{96}\epsilon.
	\end{align}
    By \eqref{oc.sec}, Lemma \ref{gradientascent} and \eqref{lem5:1}, we have
    \begin{align}\label{lem6:4}
    	\nabla^2 P(x_t) 
    	\succeq H_t - \|\nabla^2 P(x_t)-H_t\|I 
    	\succeq -\lambda_t I -\sigma\|g_t\|^{1/2} I -\epsilon_2 I 
    	\succeq -\frac{19}{12}\sqrt{L_2\epsilon}.
    \end{align}
     Then, \eqref{lem6:3} and \eqref{lem6:4} shows that $x_t$ is  an $\mathcal{O}(\epsilon,\sqrt{\epsilon})$-second-order stationary point. 
     
     For the case $\|g_t\|\leqslant \epsilon$ and $\lambda_t \geqslant \sqrt{L_2\epsilon}$, by \eqref{lem5:6}, it follows that
     \begin{align*}
     	P(x_{t+1})-P(x_t)
     	\leqslant& -\frac{1}{2}\lambda_t \|s_t\|^2 + \epsilon_1\|s_t\| + \frac{\epsilon_2}{2} \|s_t\|^2 + \frac{L_2}{6} \|s_t\|^3 \\
     	\leqslant& -\frac{1}{4}\sqrt{L_2\epsilon}\|s_t\|^2 + \epsilon_1\|s_t\| + \frac{\epsilon_2}{2} \|s_t\|^2 + \frac{L_2}{6} \|s_t\|^3.
     \end{align*}
     The rest of the proof is almost the same as in \eqref{lem5:8} and \eqref{lem5:9} by replacing $\|g_t\|$ with $\epsilon$. We omit the details here for simplicity.
\end{proof}

We are now ready to establish the iteration complexity for the GRTR algorithm.
In particular, let $\epsilon > 0$ be any given target accuracy. We denote the first iteration index to achieve an  $\mathcal{O}(\epsilon,\sqrt{\epsilon})$-second-order stationary point by 
\begin{align}\label{T}
	{T}(\epsilon) := \min\{t |\|\nabla P(x_t)\| \leqslant \xi\epsilon ~\text{and}~ \nabla^2 P(x_t) \succeq -\theta\sqrt{\epsilon} I\},
\end{align}
where $\xi$ and $\theta$ are some constants independent of $\epsilon$.
We further denote the first iteration index of the GRTR algorithm to  achieve $\|g_t\| \leqslant \epsilon$ and $\lambda_t \leqslant \sqrt{L_2\epsilon}$ by
\begin{align*}
	\tilde{T}(\epsilon) := \min\{t |\|g_t\| \leqslant \epsilon ~\text{and}~ \lambda_t \leqslant \sqrt{L_2\epsilon}\}.
\end{align*}
From Lemma \ref{lem6},  it can be known that  $x_{\tilde{T}(\epsilon)}$ is an $\mathcal{O}(\epsilon,\sqrt{\epsilon})$-second-order stationary point, indicating that $T(\epsilon) \leqslant \tilde{T}(\epsilon)$.

Let us define the following index sets to facilitate the complexity analysis:
\begin{align}
	& \mathcal{F}_{\tilde{T}(\epsilon)} =\bigg\{t < \tilde{T}(\epsilon): P(x_{t+1})-P(x_t)\leqslant -\frac{1}{128\sqrt{L_2}} \max\{\|g_t\|^{3/2}, \epsilon^{3/2}\}\bigg\}, \label{set1} \\
	& \mathcal{G}_{\tilde{T}(\epsilon)} = \left\{t < \tilde{T}(\epsilon): \|g_{t+1}\|\leqslant \frac{1}{3}\|g_t\|, t \notin \mathcal{F}_{\tilde{T}(\epsilon)}  \right\}. \label{set2}
\end{align}
According to Lemmas \ref{lem5} and \ref{lem6}, each iteration $t$ belongs to  at least one of the two sets $\mathcal{F}_{\tilde{T}(\epsilon)}$ and $\mathcal{G}_{\tilde{T}(\epsilon)}$. Therefore, our aim is to establish an upper bound for the cardinality of sets $\mathcal{F}_{\tilde{T}(\epsilon)}$ and
$\mathcal{G}_{\tilde{T}(\epsilon)}$.

 In the following lemma, we analyze $\mathcal{F}_{\tilde{T}(\epsilon)}$ by evaluating the decrease in function value.

\begin{lem}\label{lem7}
	Suppose that Assumption \ref{ass} holds. Let $\{x_t\}$ be a sequence generated by Algorithm \ref{alg:1}  with parameter settings in \eqref{lem5:1}. We have 
	\begin{align}\label{lem7:2}
		|\mathcal{F}_{\tilde{T}(\epsilon)}| \leqslant 128\sqrt{L_2}(P(x_0)-P^*) \epsilon^{-3/2}.
	\end{align}
\end{lem}

\begin{proof}
	According to Lemmas \ref{lem5} and \ref{lem6}, for any $t \geqslant 0$, we have
	\begin{align}\label{lem7:3}
		P(x_{t+1}) \leqslant P(x_t).
	\end{align}
    In addition, for any $t \in \mathcal{F}_{\tilde{T}(\epsilon)}$, we can easily obtain that
    \begin{align}\label{lem7:4}
    	 \frac{1}{128\sqrt{L_2}}\epsilon^{3/2} \leqslant \frac{1}{128\sqrt{L_2}}\max\{\|g_t\|^{3/2}, \epsilon^{3/2}\} \leqslant P(x_t)-P(x_{t+1}).
    \end{align}
     From \eqref{lem7:3}, \eqref{lem7:4} and Assumption \ref{ass}, it follows that
    \begin{align*}
    	|\mathcal{F}_{\tilde{T}(\epsilon)}|\frac{1}{128\sqrt{L_2}}\epsilon^{3/2} 
    	\leqslant& \sum_{t \in \mathcal{F}_{\tilde{T}(\epsilon)} } (P(x_t)-P(x_{t+1})) \\
    	\leqslant & \sum_{t=0}^{\tilde{T}(\epsilon)-1} (P(x_t)-P(x_{t+1})) 
    	\leqslant  P(x_0)-P^*,
    \end{align*}
    which completes the proof.
\end{proof}

Now, it remains to establish an upper bound on the index set  $|\mathcal{G}_{\tilde{T}(\epsilon)}|$.
\begin{lem}\label{lem8}
    Suppose that Assumption \ref{ass} holds. Let $\{x_t\}$ be a sequence generated by Algorithm \ref{alg:1}  with parameter settings in \eqref{lem5:1}. We have 
	\begin{equation}\label{lem8:1}
	   \vert \mathcal{G}_{\tilde{T}(\epsilon)} \vert \leqslant  \lceil\ln (G/\epsilon)/ \ln3\rceil \vert \mathcal{F}_{\tilde{T}(\epsilon)} \vert,
    \end{equation}
    where $G= \max\big\{2 + \frac{L_1}{4\sqrt{L_2}}C^{1/2} + C, 2 +\frac{L_1}{4\sqrt{L_2}}\big\}$ with  $C=\left (128\sqrt{L_2}\left (P(x_0)-P^*\right )\right)^{2/3}= \mathcal{O}(\rho^{1/3}\kappa)$, which further implies $G = \mathcal{O}(\rho^{1/3}\kappa+\ell\rho^{-1/3})$.
\end{lem}

\begin{proof}
	We first establish the uniform upper bound of the norm of  gradient. For any $t \in \mathcal{F}_{\tilde{T}(\epsilon)}$, we can conclude that 
    \begin{equation}\label{lem8:2}
    	\|g_t\| \leqslant C= \left (128\sqrt{L_2}\left (P(x_0)-P^*\right )\right)^{2/3}
    \end{equation}
    by contradiction. Suppose that \eqref{lem8:2} does not hold. Then by $t \in \mathcal{F}_{\tilde{T}(\epsilon)}$, \eqref{lem7:3} and \eqref{lem7:4}, we have
    \begin{equation*}
    	P(x_{t+1}) \leqslant P(x_t)-\frac{1}{128\sqrt{L_2}}\|g_t\|^{3/2} \leqslant P(x_0)-\frac{1}{128\sqrt{L_2}}\|g_t\|^{3/2} < P^*,
    \end{equation*}
which is a contradiction. 
    Combining \eqref{lem8:2} with \eqref{lem5:2} and \eqref{lem6:2}, for any $t \in \mathcal{F}_{\tilde{T}(\epsilon)}$, we obtain that
    \begin{align*}
    	\| g_{t+1}\| \leqslant G=\max\left\{2 + \frac{L_1}{4\sqrt{L_2}}C^{1/2} + C, 2 +\frac{L_1}{4\sqrt{L_2}}\right\}.
    \end{align*}
    Note that for any $t \in \mathcal{G}_{\tilde{T}(\epsilon)}$, the norm of the gradient will decrease linearly with factor $\frac{1}{3}$. We then deduce that the norm of gradient is  uniformly bounded by $G$.
    
    Next we give an upper bound on the index set  $|\mathcal{G}_{\tilde{T}(\epsilon)}|$. Denote the maximum number of consecutive iterates in set $ \mathcal{G}_{\tilde{T}(\epsilon)}$ as $n$. We prove $n \leqslant \lceil\ln (G/\epsilon)/ \ln(3)\rceil$ by contradiction. Assume $n > \ln (G/\epsilon)/ \ln(3)$. From $t \in \mathcal{G}_{\tilde{T}(\epsilon)}$ and $\|g_t\| \leqslant G$, we know that
    \begin{align*}
    	\|g_{t+n}\| \leqslant \big(\frac{1}{3}\big)^n \|g_t\| \leqslant \big(\frac{1}{3}\big)^n G < \epsilon,
    \end{align*}
    Then according to Lemma \ref{lem6}, $t + n \in \mathcal{F}_{\tilde{T}(\epsilon)}$, which contradicts the definition of $n$. Therefore, $n \leqslant \lceil\ln (G/\epsilon)/ \ln(3)\rceil$ and the inequality \eqref{lem8:1} follows directly.
   This completes the proof.
\end{proof}

Next, we can give results for the iteration complexity.

\begin{thm}\label{thm1}
	Suppose that Assumption \ref{ass} holds. Let $\{x_t\}$ be the sequence generated by Algorithm \ref{alg:1} with all parameters chosen the same as in Lemma \ref{lem5}. We have
\begin{align*}
	T(\epsilon)  \leqslant&~  128(\lceil\ln (G/\epsilon)/ \ln3\rceil +1)\sqrt{L_2}(P(x_0)-P^*) \epsilon^{-3/2}  + 1,\\
	\sum_{t=0}^{T(\epsilon)}N_t \leqslant 2\sqrt{\kappa}&~\ln \frac{\sqrt{\kappa+1}\|y_{-1}-y^*(x_0)\|}{A} +2\sqrt{\kappa} T(\epsilon) \ln\frac{\sqrt{\kappa+1}(A + \kappa r\sqrt{G}) }{A},
	\end{align*}
    where $G$ is defined as in Lemma \ref{lem8} and $A$ is defined as in Lemma \ref{gradientascent}.
\end{thm}
\begin{proof}
	The upper bound of $T(\epsilon)$ can be directly proved by Lemmas \ref{lem7} and \ref{lem8}, and the specific details will not be repeated here. According to Lemma \ref{gradientascent}, the total number of gradient ascent steps required satisfies
	 \begin{align*}
	 	\sum_{t=0}^{T(\epsilon)}N_t =&~ N_0 +   \sum_{t=1}^{T(\epsilon)}2\sqrt{\kappa} \ln\frac{\sqrt{\kappa+1}(A + \kappa\|s_{t-1}\|) }{A} \nonumber \\
	 	=&~N_0 + 2\sqrt{\kappa}\ln \bigg(\prod_{t=1}^{T(\epsilon)} \frac{\sqrt{\kappa+1}(A + \kappa\|s_{t-1}\|) }{A}\bigg) \nonumber \\
	 	\leqslant &~ N_0 +  2\sqrt{\kappa} T(\epsilon)\ln\bigg( \frac{1}{T(\epsilon)}\sum_{t=1}^{T(\epsilon)} \frac{\sqrt{\kappa+1}(A + \kappa\|s_{t-1}\|) }{A}\bigg). \nonumber 
	 \end{align*}
	 By further combining with  $\|s_t\| \leqslant r\max\{\epsilon^{1/2},\|g_t\|^{1/2}\}$ and $\|g_t\| \leqslant G$ in Lemma \ref{lem8}, we conclude that 
	 \begin{align*}
	 	\sum_{t=0}^{T(\epsilon)}N_t \leqslant 2\sqrt{\kappa}\ln \frac{\sqrt{\kappa+1}\|y_{-1}-y^*(x_0)\|}{A} +2\sqrt{\kappa} T(\epsilon) \ln\frac{\sqrt{\kappa+1}(A + \kappa r\sqrt{G}) }{A},
	 \end{align*}
    which completes the proof.
\end{proof}

\begin{rem}\label{remark1}
Theorem \ref{thm1} shows that the outer iteration of Algorithm \ref{alg:1} for finding an $\mathcal{O}(\epsilon,\sqrt{\epsilon})$-second-order stationary point of $P(x)$ is $\tilde{\mathcal{O}}(\ell^{1.5}\rho^{0.5}\mu^{-1.5}\epsilon^{-1.5})$. 
This is consistent with the best known iteration complexity of the second-order algorithm \cite{Chen2023ACR,Luo2022FindingSS,Yao2024TwoTR} for solving nonconvex-strongly concave minimax problems, omitting the logarithmic factor, and is better than the best iteration complexity of the first-order algorithm. In \cite{Yang2023AccI}, a first-order algorithm is proposed to obtain the second-order stationary point with an iteration complexity of $\tilde{\mathcal{O}}(\ell^{2.25}\rho^{0.25}\mu^{-1.75}\epsilon^{-1.75})$. Compared with this algorithm, Algorithm \ref{alg:1} converges slightly faster, but each step requires solving the gradient regularization trust-region subproblem. 
Theorem \ref{thm1} also implies that the total number of gradient ascent steps required is $\tilde{\mathcal{O}} (\ell^{2}\rho^{0.5}\mu^{-2}\epsilon^{-1.5})$, which matches the algorithm in \cite{Chen2023ACR,Luo2022FindingSS}.
\end{rem}

\subsection{Inexact gradient norm regularized trust-region algorithm}	
%

It should be noted that the complexity analysis in Section 2.1 is based on the exact solution of the trust-region subproblem \eqref{tr}. However, in practice, the trust-region subproblem cannot always be solved exactly. Moreover, when dealing with large-scale problems, the computational cost of obtaining such exact solutions can become very high, especially when matrix decomposition and Hessian-vector products are involved.
In this section, we propose an Inexact Gradient Norm Regularized Trust-Region (IGRTR) algorithm that allows subproblems to be solved inexactly and can be proven to have the same iteration complexity as the GRTR algorithm.

Unlike the GRTR algorithm, when updating $x$, the IGRTR algorithm approximately solves the trust-region subproblem \eqref{tr} based on the following two cases:
\begin{itemize}
	\item  When $\|g_t\| > \epsilon$, it solves the trust-region subproblem \eqref{tr} in the same way as the trust-region Newton-CG method (Algorithm 4.1 in \cite{Curtis2021TrustN}), obtaining a direction $s_t$ that satisfies either of the following conditions:
	\begin{align}\label{inc:1}
			g_t^\top s_t + \dfrac{1}{2} s_t^\top H_t s_t \leqslant -\dfrac{1}{8}\sigma  \|g_t\|^{1/2}\|s_t\|^2, ~
		\|s_t\| = r \|g_t\|^{1/2},
	\end{align}
	or
	\begin{align}\label{inc:2}
		\begin{cases}
		g_t^\top s_t + \dfrac{1}{2} s_t^\top H_t s_t \leqslant -\dfrac{1}{8}\sigma  \|g_t\|^{1/2}\|s_t\|^2, ~
		\|s_t\| < r \|g_t\|^{1/2}, \\
		\|( H_t + \sigma\|g_t\|^{1/2}I)s_t+g_t \| \leqslant \dfrac{1}{2} \min \bigg\{\| g_t \|, \dfrac{1}{2}\sigma\|g_t\|^{1/2} \| s_t \|\bigg\}   ;
	   \end{cases}
	\end{align}
   \item When $\|g_t\| \leqslant \epsilon$, the algorithm applies the randomized Lanczos procedure (Algorithm 3.2 in \cite{Curtis2021TrustN}, \cite{Kuczynski1992Estimating}) to $H_t$, either $\lambda_{\min}(H_t)\geqslant -\dfrac{1}{2}\sigma \epsilon^{1/2}$, or else obtains a direction $s_t$ that satisfies
   \begin{equation}\label{inc:3}
   	g_t^\top s_t + \frac{1}{2} s_t^\top H_t s_t \leqslant -\frac{1}{8}\sigma \epsilon^{1/2}\|s_t\|^2, ~\|s_t\| = r \epsilon^{1/2}.
   \end{equation}
\end{itemize}
Note that from Lemmas \ref{gradientascent} and \eqref{lem5:5}, it is easy to verify that an exact solution to the trust-region subproblem \eqref{tr} satisfies these conditions.
By replacing Step 4 in the GRTR algorithm with the above procedure, the detailed IGRTR algorithm for solving \eqref{P} can be immediately obtained.

In the next section, we will show that in order to solve \eqref{P} for an $\mathcal{O}(\epsilon,\sqrt{\epsilon})$-second-order stationary point of $P(x)$, the outer iteration complexity of the IGRTR algorithm is also $\tilde{\mathcal{O}}(\ell^{1.5}\rho^{0.5}\mu^{-1.5}\epsilon^{-1.5})$, and the required Hessian-vector product is at most $\tilde{\mathcal{O}}(\ell^{0.5}\rho^{0.25}\kappa^{1.75}\epsilon^{-1.75})$.

\subsection{Complexity analysis}
In the following lemmas, we give key inequalities on the descent of the function value and the norm of the gradient for two different cases $\|g_t\| \geqslant \epsilon $ and $\|g_t\| \leqslant \epsilon$.

\begin{lem}\label{in1}
	Suppose that Assumption \ref{ass} holds. Let $\{x_t\}$ be the sequence generated by the IGRTR algorithm. Set
	\begin{align}\label{in1:1}
		\sigma=\frac{8\sqrt{L_2}}{9}, r = \frac{1}{3\sqrt{L_2}}, \epsilon_1 \leqslant \min \bigg\{\frac{1}{162},\frac{\sqrt{L_2}}{162L_1}\bigg\}\epsilon^{3/2}, \epsilon_2 \leqslant \frac{\sqrt{L_2}}{27}\epsilon^{1/2},
	\end{align}
	and choose $\eta_y, \theta, N_t$ the same as in Lemma \ref{gradientascent}. For the case $\| g_t \| > \epsilon$, if $\| s_t\| = r \|g_t\|^{1/2}$, we have
	\begin{align}\label{in1:2}
		P(x_{t+1}) - P(x_t) \leqslant -\frac{1}{486\sqrt{L_2}}\|g_t\|^{3/2}, ~\|g_{t+1}\| \leqslant 2 + \frac{L_1}{3\sqrt{L_2}}\|g_t\|^{1/2} + \|g_t\|;
	\end{align}
	if $\| s_t\| < r \|g_t\|^{1/2}$, we have
	\begin{align}\label{in1:3}
		P(x_{t+1}) \leqslant P(x_t),~\|g_{t+1}\| \leqslant \frac{1}{2} \|g_t\|.
	\end{align}
\end{lem}
\begin{proof}
	According to \eqref{lem5:4}, \eqref{inc:1} and \eqref{inc:2}, we can get
	\begin{align}\label{in1:5}
		P(x_{t+1}) - P(x_t) \leqslant -\frac{1}{8}\sigma \|g_t\|^{1/2} \|s_t\|^2 + \epsilon_1\|s_t\| + \frac{\epsilon_2}{2} \|s_t\|^2 + \frac{L_2}{6} \|s_t\|^3.
	\end{align}
	Substituting \eqref{in1:1} and $\|s_t\|=r\|g_t\|^{1/2}$ into the above inequality, we get the first inequality of \eqref{in1:2}. The second inequality of \eqref{in1:2} is the same as that in \eqref{lem5:9}.
	
	Next, we consider the case where $\| s_t\| < r \|g_t\|^{1/2}$. 	To bound the right-hand side of \eqref{in1:5}, we first give some bounds on $\|g_t\|$.
	By combining \eqref{inc:2} with $\|H_t\| \leqslant \|H_t-\nabla^2 P(x_t)\|+ \|\nabla^2 P(x_t)\| \leqslant L_1+\epsilon_2$, we can easily obtain
	\begin{align*}
		\|g_t\| \leqslant&~ \|(H_t + \sigma\|g_t\|^{1/2}I)s_t + g_t\| + \|(H_t + \sigma\|g_t\|^{1/2}I)s_t\| \\
		\leqslant&~ \frac{1}{2}\|g_t\| + \big(L_1 + \epsilon_2 + \sigma\|g_t\|^{1/2}\big)\|s_t\|.
	\end{align*}
	This means
	\begin{align}\label{in1:9}
		\epsilon^{3/2}\leqslant \epsilon^{1/2}\|g_t\| \leqslant 2(L_1+\sigma)\|g_t\|^{1/2}\|s_t\| + 2\epsilon_2\|s_t\|.
	\end{align}
	The rest of the proof is the same as Lemma \ref{lem5}. We omit it here for simplicity.

\end{proof}

\begin{lem}\label{in2}
	Suppose that Assumption \ref{ass} holds. Let $\{x_t\}$ be the sequence generated by the IGRTR algorithm, with all parameters chosen the same as in Lemma \ref{in1}. Then, for the case of $\|g_t\|\leqslant \epsilon$, if $\lambda_{\min}(H_t)\geqslant -\frac{\sigma}{2} \epsilon^{1/2}$, then we have found an $\mathcal{O}(\epsilon,\sqrt{\epsilon})$-second-order stationary point $x_t$; otherwise, we have
	\begin{align}\label{in2:1}
		P(x_{t+1}) - P(x_t) \leqslant -\frac{1}{486\sqrt{L_2}}\epsilon^{3/2},
		~\|g_{t+1}\| \leqslant 2 + \frac{L_1}{3\sqrt{L_2}}\epsilon^{1/2}.
	\end{align}
\end{lem}
Combined with \eqref{inc:3}, the following proof is basically the same as the proof of Lemma \ref{lem6} and will not be repeated here.

Let $\epsilon > 0$ be any given target accuracy. We denote the first iteration of the IGRTR algorithm to obtain $\|g_t\| \leqslant \epsilon $ and $\lambda_{\min}(H_t) \geqslant -\frac{4}{9}\sqrt{L_2\epsilon}$ as
\begin{align*}
	\hat{T}(\epsilon) := \min\left\{t |\|g_t\| \leqslant \epsilon ~\text{and}~ \lambda_{\min}(H_t) \geqslant -\dfrac{4}{9}\sqrt{L_2\epsilon}\right\}.
\end{align*}
Then, by Lemma \ref{in2}, we know that $x_{\hat{T}(\epsilon)}$ is already an $\mathcal{O}(\epsilon,\sqrt{\epsilon})$-second-order stationary point.

Let us define the following index sets to facilitate the complexity analysis,
\begin{align}
	& \mathcal{F}_{\hat{T}(\epsilon)} = \left\{t < \hat{T}(\epsilon): P(x_{t+1})-P(x_t)\leqslant -\frac{1}{486\sqrt{L_2}} \max\{\|g_t\|^{3/2}, \epsilon^{3/2}\}\right\},  \\
	& \mathcal{G}_{\hat{T}(\epsilon)} = \left\{t < \hat{T}(\epsilon): \|g_{t+1}\|\leqslant \frac{1}{2}\|g_t\|,  t \notin \mathcal{F}_{\hat{T}(\epsilon)} \right\}.
\end{align}
From Lemmas \ref{in1} and \ref{in2}, we  conclude that each iteration $t$ belongs to at least one of the above two sets $\mathcal{F}_{\hat{T}(\epsilon)}$ and $\mathcal{G}_{\hat{T}(\epsilon)}$.

Similar to the proof of Theorem \ref{thm1} in Section 2, we give an upper bound on $\hat{T}(\epsilon)$ in the following theorem.

\begin{thm}\label{inthm1}
	Suppose that Assumption \ref{ass} holds. Let $\{x_t\}$ be a sequence generated by the IGRTR algorithm with all parameters chosen the same as in Lemma \ref{in1}.  We have
	\begin{align*}
		\hat{T}(\epsilon)  \leqslant&~  \hat{C}_{1}\epsilon^{-3/2}\ln\epsilon^{-1} + \hat{C}_{2}\epsilon^{-3/2}+1=\tilde{\mathcal{O}}(\ell^{1.5}\rho^{0.5}\mu^{-1.5}\epsilon^{-1.5}),\\
		\sum_{t=0}^{\hat{T}(\epsilon)}N_t \leqslant 2\sqrt{\kappa}\ln& \frac{\sqrt{\kappa+1}\|y_{-1}-y^*(x_0)\|}{A} +2\sqrt{\kappa} \hat{T}(\epsilon) \ln\frac{\sqrt{\kappa+1}(A + \kappa r\sqrt{\hat{G}}) }{A},
	\end{align*}
	where
    $\hat{C}_{1}=486\sqrt{L_2}(P(x_0)-P^*)/\ln2$, $\hat{C}_{2}=(\ln\hat{G}+\ln2)\hat{C}_{1}$,
	$\hat{G}:= \max\big\{ 2 + \frac{L_1}{3\sqrt{L_2}}\hat{C}^{1/2} + \hat{C}, 2 +\frac{L_1}{3\sqrt{L_2}}\big\}$ with $\hat{C}:=\left (486\sqrt{L_2}\left (P(x_0)-P^*\right )\right)^{2/3}$, and $A$ is defined as in Lemma \ref{gradientascent}.
\end{thm}

Next, we analyze the Hessian-vector product required by the IGRTR algorithm.


\begin{cor}\label{cor1}
Suppose that Assumption \ref{ass} holds. The matrix inverse computation $(\nabla_{y y}^{2}f(x_t,y_t))^{-1}$ in $H_t$ can be approximated by $\tilde{\mathcal{O}}(\sqrt{\kappa})$ terms of Chebyshev polynomials as described in \cite{Luo2022FindingSS}.
Let $\{x_t\}$ be the sequence generated by the IGRTR algorithm, and all other parameters are chosen the same as in Lemma \ref{in1}.
Then it can be shown that $x_{\hat{T}(\epsilon)}$ is an $\mathcal{O}(\epsilon,\sqrt{\epsilon})$-second-order stationary point with high probability guarantee. Therefore, the total number of Hessian-vector products  required is at most $\tilde{\mathcal{O}}(\ell^{2.25}\rho^{0.25}\mu^{-1.75}\epsilon^{-1.75})$.
\end{cor}
\begin{proof}
	Before analyzing the Hessian vector products required by the IGRTR algorithm, we first show that replacing $(\nabla_{y y}^{2}f(x_t,y_t))^{-1}$ with an approximation does not affect the iteration complexity bound in Theorem \ref{inthm1}. Denote 
    \begin{align*}
    	\hat{H}_t:= \nabla_{x x}^{2} f(x_t,y_t)+\nabla_{x y}^{2} f(x_t,y_t) C_t \nabla_{y x}^{2} f(x_t,y_t) ~\text{with}~  C_t\approx -(\nabla_{y y}^{2}f(x_t,y_t))^{-1}.
    \end{align*}
    By the choice of $N_t$ in Lemma \ref{gradientascent} and Lemma 8 in \cite{Luo2022FindingSS}, it can be proved that
    \begin{equation*}
    	\|\hat{H}_t-\nabla^2 P(x_t)\| \leqslant \|\hat{H}_t-H_t\| + \|H_t-\nabla^2 P(x_t)\| \leqslant \epsilon_2,
    \end{equation*}
    This shows that the iteration complexity result in Theorem \ref{inthm1} is not affected when $\hat{H}_t$ is used instead of $H_t$. In addition, for any vector $u$, we can obtain $\hat{H}_tu$ by calling $\tilde{\mathcal{O}}(\sqrt{\kappa})$ Hessian-vector products (Appendix E \cite{Luo2022FindingSS}).
    
  Note that by Theorem \ref{inthm1}, the outer iteration complexity is $\tilde{\mathcal{O}}(\ell^{1.5}\rho^{0.5}\mu^{-1.5}\epsilon^{-1.5})$. In each outer iteration, according to Lemmas 3.1 and 4.1 in \cite{Curtis2021TrustN} and Lemma 2 in \cite{Royer2020ANewton}, approximately solving the trust-region subproblem \eqref{tr} requires at most $\min\big\{n, \tilde{\mathcal{O}}(L_1^{0.5}L_2^{-0.25}\epsilon^{-0.25})\big\}$ iterations to obtain a direction $s_t$ that satisfies conditions \eqref{inc:1}, \eqref{inc:2}, and \eqref{inc:3} with probability $1-\xi$, where $\xi$ is the failure probability of the Lanczos procedure.
  Since the failure probability $\xi$ is only within a "logarithmic factor" of the complexity bound, we can use a very small $\xi=\hat{\xi}/(\hat{C}_{1}\epsilon^{-3/2}\ln\epsilon^{-1} + \hat{C}_{2}\epsilon^{-3/2}+1)$, where $\hat{\xi} \in (0,1)$, $\hat{C}_{1}$ and $\hat{C}_{2}$ are defined in Theorem \ref{inthm1}.
  Therefore, $x_{\hat{T}(\epsilon)}$ is an expected second-order stationary point with probability $1-\hat{\xi}$, and the total number of Hessian-vector products required is
  \begin{align*}
  	\tilde{\mathcal{O}}(\sqrt{\kappa})\hat{T}(\epsilon)\min\big\{n, \tilde{\mathcal{O}}(L_1^{0.5}L_2^{-0.25}\epsilon^{-0.25})\big\}
  	\leqslant \tilde{\mathcal{O}}(\ell^{0.5}\rho^{0.25}\kappa^{1.75}\epsilon^{-1.75}).
  \end{align*}
  The proof is complete.
\end{proof}

\begin{rem}
Compared with Theorem \ref{thm1}, the above Theorem \ref{inthm1} and Corollary \ref{cor1} show that the number of outer iterations and gradient ascent iterations of the IGRTR algorithm are consistent with the order of the GRTR algorithm. Compared to the IMCN algorithm \cite{Luo2022FindingSS} and the ICLM algorithm \cite{Chen2023ACR}, both of which require $\tilde{\mathcal{O}}(\ell^{2.5}\mu^{-1.5}\epsilon^{-2})$ Hessian-vector products, the IGRTR algorithm only requires $\tilde{\mathcal{O}}(\ell^{2.25}\rho^{0.25}\mu^{-1.75}\epsilon^{-1.75})$ Hessian-vector products, which has a better complexity bound in terms of the order of $\epsilon$. 
\end{rem}

\section{A Levenberg-Marquardt algorithm with negative curvature correction}
In the classic nonlinear minimization optimization,  when the Hessian matrix is not positive definite,
the negative curvature  direction can be used to correct the iteration direction of the algorithm to converge to a second-order stationary point~\cite{Gratton2023YetAF}.
In this section, we further propose a Levenberg-Marquardt (LM) algorithm with a gradient norm regularization coefficient \cite{Mischenko2023RegularizedNM}, which uses the negative curvature direction to correct the iteration direction (LMNegCur) to solve \eqref{P} without solving the trust-region subproblem at each iteration. The LMNegCur algorithm consists of the following two important algorithmic components at each iteration:
\begin{itemize}
	\item Gradient ascent update: for given $x_t$, for a given $x_t$,  run $N_t$ iterations of Nesterov's accelerated gradient ascent step to obtain an approximated maximizer  
	\begin{align*}
		y_t \approx y^*(x_t) = \arg\max_{y\in \mathbb{R}^m}f(x,y).
	\end{align*}
	\item LM or negative curvature update: for a given pair $(x_t,y_t)$, if $\lambda_{\min}(H_t) \leqslant -\frac{1}{2}\sqrt{L_2\max\{\|g_t\|,\epsilon\}}$, perform a negative curvature update
	\begin{align*}
		s_t= \sqrt{\max\{\|g_t\|,\epsilon\}/L_2}u_t, ~\text{where}~ g_t^{\top}u_t \leqslant0, \|u_t\|=1, ~\lambda_{\min}(H_t)=u_t^{\top}H_tu_t;
	\end{align*}
	otherwise, if $\|g_t\|\geqslant \epsilon$ and $\lambda_{\min}(H_t)\geqslant -\dfrac{1}{2}\sqrt{L_2\|g_t\|}$, perform LM update with a gradient norm regularization coefficient
	\begin{align*}
		(H_t+ \sqrt{L_2\|g_t\|}I)s_t = - g_t,	
	\end{align*}
   where $g_t=\nabla_x f(x_t,y_t)$ and  $H_t=H(x_t,y_t)$ are inexact first-order and second-order information of $P(x)$ at $x_t$, respectively.

\end{itemize}
The detailed algorithm is formally stated in Algorithm \ref{alg:2}.

\begin{algorithm}[th]
	\caption{A LMNegCur Algorithm}
	\label{alg:2}
	\begin{algorithmic}
		\STATE{\textbf{Step 1}: Input $x_0,y_{-1},\eta_y,\theta, N_t,\epsilon$; Set $t=0$.}
			\STATE{\textbf{Step 2}: Update $y_t$: }
			\STATE{\quad\textbf{(a)}: Set $k=0$, $y_k^t=\tilde{y}_k^t=y_{t-1}$.}
			\STATE{\quad\textbf{(b)}: Update $y_k^t$ and $\tilde{y}_k^t$:
				\begin{align}
				y_{k+1}^t = \tilde{y}_{k}^t + \eta_y \nabla_y f(x_t, \tilde{y}_k^t), ~
					\tilde{y}_{k+1}^t = y_{k+1}^t +  \theta(y_{k+1}^t-y_{k}^t).
				\end{align}
			}
			\STATE{\quad\textbf{(c)}: If $k \geqslant N_t-1$, set $y_t={y}_{N_t}^t$; otherwise, set $k=k+1, $ go to Step 2(b).}
		\STATE{\textbf{Step 3}: Set  $g_t=\nabla_x f(x_t,y_t)$ and  $H_t=H(x_t,y_t)$}.
		\STATE{\textbf{Step 4}: Compute an eigenpair $(\lambda_{\min}(H_t), u_t)$ such that 
			\begin{align}\label{alg2:3}
			\lambda_{\min}(H_t)=u_t^{\top}H_tu_t, \|u_t\|=1, ~g_t^{\top}u_t \leqslant0.
			\end{align}
			 \quad\quad\quad~ If $\lambda_{\min}(H_t) \leqslant -\frac{1}{2}\sqrt{L_2\max\{\|g_t\|,\epsilon\}}$, 
		\begin{align}\label{alg2:4}
			s_t= \sqrt{\max\{\|g_t\|,\epsilon\}/L_2}u_t;
		\end{align}
		\quad\quad\quad~ otherwise, go to Step 5.}
		\STATE{\textbf{Step 5}: If $\|g_t\|\geqslant \epsilon$,
		\begin{align}\label{alg2:2}
			(H_t+ \sqrt{L_2\|g_t\|}I)s_t = -g_t;	
	    \end{align}
         	\quad\quad\quad~~ otherwise, stop and output $(x_t,y_t)$.}
		\STATE{\textbf{Step 6}: Update $x_t$: $x_{t+1} = x_t+s_t.$ Set $t=t+1, $ go to Step 2(a).}
	\end{algorithmic}
\end{algorithm}

Note that for the case $\|g_t\|\leqslant \epsilon, ~\lambda_{\min}(H_t) \geqslant -\frac{1}{2}\sqrt{L_2\epsilon}$, we have already achieved an $\mathcal{O}(\epsilon,\sqrt{\epsilon})$-second-order stationary point of $P(x)$, which will be proved later. Differing from the trust-region type \cite{Yao2024TwoTR} and cubic regularization type algorithms \cite{Luo2022FindingSS,Chen2023ACR}, the LMNegCur algorithm is free of further inner iterative processes in the involved outer iteration and only requires at most one negative curvature computation and one solution of a positive definite linear system. In addition, the hard-case of trust-region subproblem and cubic regularization subproblem, in which the gradient $g_t$ is orthogonal to the eigenspace of the Hessian $H_t$ corresponding to its minimum eigenvalue, does not present a serious challenge.
The minimum eigenvalue of $H_t$ and the corresponding curvature direction can be retrieved via a Lanczos procedure \cite{Kuczynski1992Estimating}. Since $\|H_t\|\leqslant L_1$, it was shown in \cite{Royer2018ComplexityAO} (Lemma $9$) that after at most $n$ (the dimension of $x$) iterations, the procedure obtains a unit vector $u$ such that $u^{\top}H_tu =\lambda_{\min}(H_t)$ with probability $1$. 
 As also noted in Chapter $10$ of \cite{Golub2013 Matrix}, for large and sparse problem, information about the extremal eigenvalues of matrix tends to emerge quite early in the iteration of the Lanczos procedure. Then, using the smallest eigenvalue of the Hessian does not incur much cost in such case.

In the following subsection, we will establish the iteration complexity of the LMNegCur algorithm to obtain an $\mathcal{O}(\epsilon,\sqrt{\epsilon})$-second-order stationary point of $P(x)$.

\subsection{Complexity analysis}
%
In the following lemmas, we first prove some key inequalities about functions and gradients in different cases, namely, Case 1 \big($\|g_t\|\geqslant \epsilon$ and $\lambda_{\min}(H_t) \geqslant -\dfrac{1}{2}\sqrt{L_2\|g_t\|}$\big), Case 2 \big($\lambda_{\min}(H_t) \leqslant -\dfrac{1}{2}\sqrt{L_2\max\{\|g_t\|,\epsilon\}}$\big), and Case 3 \big($\|g_t\|\leqslant \epsilon$ and $\lambda_{\min}(H_t) \geqslant -\dfrac{1}{2}\sqrt{L_2\epsilon}$\big).

\begin{lem}\label{lem3.1}
	Suppose that Assumption \ref{ass} holds. Let $\{x_t\}$ be a sequence generated by Algorithm \ref{alg:2}. By choosing 
	\begin{align}\label{lem3.1:1}
	 \epsilon_1  =\min\bigg\{\frac{1}{36}, \frac{\sqrt{L_2}}{12L_1}\bigg\}\epsilon^{3/2}, ~\epsilon_2 = \frac{\sqrt{L_2}}{18}\epsilon^{1/2} 
	\end{align} 
  in Lemma \ref{gradientascent}  and choosing $\eta_y,\theta, N_t$ in \eqref{N_t}. For the case $\|g_t\|\geqslant \epsilon$, if $\lambda_{\min}(H_t) \geqslant -\dfrac{1}{2}\sqrt{L_2\|g_t\|}$, we have
    \begin{align}
    	~\|g_{t+1}\| \leqslant&~ \frac{41}{9}\|g_t\|, \label{lem3.1:2} \\
    	P(x_{t+1})  \leqslant P(x_t)- \dfrac{5}{18}\sqrt{L_2\|g_t\|}\bigg(&\frac{-23\sqrt{L_2\|g_t\|} + \sqrt{529L_2\|g_t\| + 648L_2\|g_{t+1}\|}}{18L_2}\bigg)^2 \label{lem3.1:3};
    \end{align}
    If $\|g_{t+1}\| \geqslant \frac{1}{2}\|g_t\|$ holds, then we further have
    \begin{align}\label{lem3.1:4}
    	P(x_{t+1}) \leqslant P(x_t)-\frac{5}{162\sqrt{L_2}}\|g_t\|^{3/2}.
    \end{align}
\end{lem}
\begin{proof}
	By the same proof as in \eqref{lem5:4}, we have
	\begin{align*}
		P(x_{t+1}) - P(x_t) 
		\leqslant g_t^\top s_t + \frac{1}{2} s_t^\top H_t s_t + \epsilon_1\|s_t\| 
		+ \frac{\epsilon_2}{2} \|s_t\|^2 + \frac{L_2}{6} \|s_t\|^3.
	\end{align*}
	By plugging \eqref{alg2:2} into the above relation and using the assumption that $\lambda_{\min}(H_t) \geqslant -\dfrac{1}{2}\sqrt{L_2\|g_t\|}$, we obtain that
	\begin{align}\label{lem3.1:5}
		&~P(x_{t+1}) - P(x_t) \nonumber\\
		\leqslant & -\frac{1}{2} s_t^{\top} (H_t + \frac{1}{2}\sqrt{L_2\|g_t\|}I ) s_t - \frac{3}{4}  \sqrt{L_2\|g_t\|}\|s_t\|^2 + \epsilon_1\|s_t\| 
		+ \frac{\epsilon_2}{2} \|s_t\|^2 + \frac{L_2}{6} \|s_t\|^3 \nonumber \\
		\leqslant & - \frac{3}{4}  \sqrt{L_2\|g_t\|}\|s_t\|^2  + \epsilon_1\|s_t\| + \frac{\epsilon_2}{2} \|s_t\|^2 + \frac{L_2}{6} \|s_t\|^3.
	\end{align} 
    To further bound the right-hand side of \eqref{lem3.1:5}, we first give some bounds on $\|s_t\|$.
   By combining \eqref{alg2:2} with $\lambda_{\min}(H_t) \geqslant -\dfrac{1}{2}\sqrt{L_2\|g_t\|}$ and $\|H_t\| \leqslant \ell(1+\kappa)=L_1$, we then obtain that
    \begin{align}\label{lem3.1:6}
        \frac{\|g_t\|}{L_1 + \sqrt{L_2\|g_t\|}}\leqslant \|s_t\| \leqslant \|(H_t+\sqrt{L_2\|g_t\|}I)^{-1}\|\|g_t\|\leqslant 2\sqrt{\dfrac{\|g_t\|}{L_2}},
    \end{align}
    which implies that 
     \begin{align}\label{lem3.1:7}
    	 (L_1+ \sqrt{L_2})\sqrt{\|g_t\|}\|s_t\|\geqslant \epsilon^{1/2}\|g_t\| \geqslant \epsilon^{3/2}.
    \end{align} 
   Combining \eqref{lem3.1:5} with \eqref{lem3.1:6} and \eqref{lem3.1:7}, and by \eqref{lem3.1:1}, we conclude that
    \begin{align}\label{lem3.1:8}
    	&~P(x_{t+1}) - P(x_t) \nonumber\\
    	\leqslant&  - \frac{5}{12}\sqrt{L_2\|g_t\|}\|s_t\|^2 + \min\bigg\{\frac{1}{36}, \frac{\sqrt{L_2}}{12L_1}\bigg\}\epsilon^{3/2}\|s_t\| + \frac{\sqrt{L_2}}{36}\epsilon^{1/2}\|s_t\|^2 \nonumber\\
    	\leqslant& - \left(\frac{7\sqrt{L_2}}{18}-\min\bigg\{\frac{1}{36}, \frac{\sqrt{L_2}}{12L_1}\bigg\}\frac{\epsilon^{3/2}}{\sqrt{\|g_t\|}\|s_t\|}\right)\sqrt{\|g_t\|}\|s_t\|^2, \nonumber\\
    	\leqslant& - \dfrac{5}{18}\sqrt{L_2\|g_t\|}\|s_t\|^2.
    \end{align}
   By \eqref{alg2:2} and using an argument similar to the proof of \eqref{lem5:12}, we have
    \begin{align*}
    	\|g_{t+1}\| \leqslant 2\epsilon_1 + \frac{L_2}{2}\|s_t\|^2 + \epsilon_2\|s_t\| + \sqrt{L_2\|g_t\|}\|s_t\|.
    \end{align*}
    By further combining the above relation with \eqref{lem3.1:1} and \eqref{lem3.1:7}, we obtain that
    \begin{align}\label{lem3.1:9}
    	\|g_{t+1}\|\leqslant\frac{23}{18}\sqrt{L_2\|g_t\|}\|s_t\| +\frac{L_2}{2}\|s_t\|^2.
    \end{align}
    Then, \eqref{lem3.1:2} holds directly by combining \eqref{lem3.1:9} with \eqref{lem3.1:6}. On the other hand, \eqref{lem3.1:9} implies that
    \begin{align}\label{lem3.1:10}
    	\|s_t\| \geqslant \frac{-23\sqrt{L_2\|g_t\|} + \sqrt{529L_2\|g_t\| + 648L_2\|g_{t+1}\|}}{18L_2}.
    \end{align} 
    Then, \eqref{lem3.1:3} holds directly by combining \eqref{lem3.1:8} with \eqref{lem3.1:10}, and
    \eqref{lem3.1:4} holds by further combining \eqref{lem3.1:3} with $\|g_{t+1}\| \geqslant \frac{1}{2}\|g_t\|$.
    The proof is completed.
\end{proof}

\begin{lem}\label{lem3.2}
	Suppose that Assumption \ref{ass} holds. Let $\{x_t\}$ be a sequence generated by Algorithm \ref{alg:2} with parameters chosen the same as in Lemma \ref{lem3.1}. For the case  $\lambda_{\min}(H_t) \leqslant -\dfrac{1}{2}\sqrt{L_2\max\{\|g_t\|,\epsilon\}}$, we have
	\begin{align}
		&P(x_{t+1}) \leqslant P(x_t)-\frac{1}{36\sqrt{L_2}}\max\{\|g_t\|^{3/2},\epsilon^{3/2}\}, \label{lem3.2:2}\\ 
		&\|g_{t+1}\| \leqslant 1 + \frac{L_1}{\sqrt{L_2}}\sqrt{\max\{\|g_t\|,\epsilon\}} + \max\{\|g_t\|,\epsilon\} \label{lem3.2:3}.
	\end{align}
\end{lem}

\begin{proof}
    By \eqref{alg2:3} and  $\lambda_{\min}(H_t) \leqslant -\dfrac{1}{2}\sqrt{L_2\max\{\|g_t\|,\epsilon\}}$, we have 
    \begin{align}\label{lem3.2:4}
    	\|s_t\|= \sqrt{\frac{\max\{\|g_t\|,\epsilon\}}{L_2}}, ~g_t^{\top}s_t \leqslant0, ~g_t^\top s_t + \frac{1}{2} s_t^\top H_t s_t \leqslant - \frac{1}{4\sqrt{L_2}}\max\{\|g_t\|^{3/2},\epsilon^{3/2}\}.
    \end{align}
 Combining \eqref{lem5:4} and \eqref{lem5:9} with \eqref{lem3.2:4}, and by \eqref{lem3.1:1}, we obtain 
    \begin{align*}
    	P(x_{t+1}) \leqslant&~  P(x_t)- \frac{1}{12\sqrt{L_2}}\max\{\|g_t\|^{3/2},\epsilon^{3/2}\} + \epsilon_1 \sqrt{\frac{\max\{\|g_t\|,\epsilon\}}{L_2}} + \frac{\epsilon_2 }{2}\frac{\max\{\|g_t\|,\epsilon\}}{L_2} \\
    	\leqslant&~  P(x_t)- \frac{1}{36\sqrt{L_2}}\max\{\|g_t\|^{3/2},\epsilon^{3/2}\},
    \end{align*}
    and
    \begin{align*}
    	\|g_{t+1}\|
    	\leqslant&~ 2\epsilon_1 + L_1\|s_t\| + \|g_t\| \leqslant 1 + \frac{L_1}{\sqrt{L_2}}\sqrt{\max\{\|g_t\|,\epsilon\}} + \max\{\|g_t\|,\epsilon\}.
    \end{align*}
    The proof is then completed.
\end{proof}


\begin{lem}\label{lem3.4}
	Suppose that Assumption \ref{ass} holds. Let $\{x_t\}$ be a sequence generated by Algorithm \ref{alg:2} with parameters chosen the same as in Lemma \ref{lem3.1}. For the case $\|g_t\|\leqslant \epsilon$ and $\lambda_{\min}(H_t) \geqslant -\dfrac{1}{2}\sqrt{L_2\epsilon}$,  we have already found an $\mathcal{O}(\epsilon,\sqrt{\epsilon})$-second-order stationary point $x_t$.
\end{lem}
\begin{proof}
	By $\|g_t\| \leqslant \epsilon$, $\lambda_{\min}(H_t) \geqslant -\dfrac{1}{2}\sqrt{L_2\epsilon}$, Lemma \ref{gradientascent} and \eqref{lem3.1:1}, we have
	\begin{align*}
		\|\nabla P(x_t)\| \leqslant \|\nabla P(x_t)-g_t\| + \|g_t\| \leqslant \frac{37}{36}\epsilon,
	\end{align*}
	 and 
	\begin{align*}
		\nabla^2 P(x_t) 
		\succeq H_t - \|\nabla^2 P(x_t)-H_t\|I 
		\succeq \lambda_{\min}(H_t)-\epsilon_2 I 
		\succeq -\frac{5}{9}\sqrt{L_2\epsilon}.
	\end{align*}
    The proof is then completed.
\end{proof}

We are now ready to establish the iteration complexity for the LMNegCur algorithm.
In particular, let $\epsilon > 0$ be any given target accuracy. We denote the first iteration index to achieve an  $\mathcal{O}(\epsilon,\sqrt{\epsilon})$-second-order stationary point the same as \eqref{T}. We further denote the first iteration index of the LMNegCur algorithm to achieve $\|g_t\| \leqslant \epsilon $ and $\lambda_{\min}(H_t) \geqslant -\frac{1}{2}\sqrt{L_2\epsilon}$ by 
\begin{align*}
	\bar{T}(\epsilon) := \min\left\{t |\|g_t\| \leqslant \epsilon ~\text{and}~ \lambda_{\min}(H_t) \geqslant -\frac{1}{2}\sqrt{L_2\epsilon}\right\}.
\end{align*}
Then, from Lemma \ref{lem3.4}, we know that  $x_{\bar{T}(\epsilon)}$ is already an $\mathcal{O}(\epsilon,\sqrt{\epsilon})$-second-order stationary point, indicating that $T(\epsilon) \leqslant \bar{T}(\epsilon)$.

Let us define the following index sets to facilitate the complexity analysis,
\begin{align}
	& \mathcal{F}_{\bar{T}(\epsilon)} = \left\{t < \bar{T}(\epsilon): P(x_{t+1})-P(x_t)\leqslant -\frac{1}{36\sqrt{L_2}} \max\{\|g_t\|^{3/2}, \epsilon^{3/2}\}\right\},  \\
	& \mathcal{G}_{\bar{T}(\epsilon)} = \left\{t < \bar{T}(\epsilon): \|g_{t+1}\|\leqslant \frac{1}{2}\|g_t\|,  t \notin \mathcal{F}_{\bar{T}(\epsilon)} \right\}.
\end{align}
From Lemmas \ref{lem3.1} and \ref{lem3.2}, we  conclude that each iteration $t$ belongs to at least one of the above two sets $\mathcal{F}_{\bar{T}(\epsilon)}$ and $\mathcal{G}_{\bar{T}(\epsilon)}$.

Similar to the proof in Section 2, we give an upper bound on $T(\epsilon)$ in the following theorem.

\begin{thm}\label{thm3.1}
	Suppose that Assumption \ref{ass} holds. Let $\{x_t\}$ be a sequence generated by Algorithm \ref{alg:2} with parameters chosen the same as in Lemma \ref{lem3.1}. We have
	\begin{align*}
		T(\epsilon)  &\leqslant  \bar{C}_{1}\epsilon^{-3/2}\ln\epsilon^{-1} + \bar{C}_{2}\epsilon^{-3/2}+1=\tilde{\mathcal{O}}(\ell^{1.5}\rho^{0.5}\mu^{-1.5}\epsilon^{-1.5}),\\
		\sum_{t=0}^{T(\epsilon)}N_t \leqslant 2\sqrt{\kappa}&~ \ln \frac{\sqrt{\kappa+1}\|y_{-1}-y^*(x_0)\|}{A} +2\sqrt{\kappa} T(\epsilon) \ln\frac{\sqrt{\kappa+1}(A + 2\kappa \sqrt{\bar{G}}/\sqrt{L_2}) }{A},
	\end{align*}
    where 
    $\bar{C}_{1}=36\sqrt{L_2}(P(x_0)-P^*)/\ln2$, $\bar{C}_{2}=(\ln\bar{G}+\ln2)\bar{C}_{1}$,
    $\bar{G}:= \max\big\{\frac{41}{9}\bar{C}, 1 + \frac{L_1}{\sqrt{L_2}}\bar{C}^{1/2} + \bar{C}, 2 +\frac{L_1}{\sqrt{L_2}}\big\}$ with $\bar{C}:=\left (36\sqrt{L_2}\left (P(x_0)-P^*\right )\right)^{2/3}$, and $A$ is defined as in Lemma \ref{gradientascent}.
\end{thm}
We omit the proofs here since they are almost identical to those in Theorem \ref{thm1}.

\begin{rem}
	Theorem \ref{thm3.1} shows that the iteration complexity of Algorithm \ref{alg:2} to find an $\mathcal{O}(\epsilon,\sqrt{\epsilon})$-second-order stationary point of $P(x)$ is $\tilde{\mathcal{O}}(\ell^{1.5}\rho^{0.5}\mu^{-1.5}\epsilon^{-1.5})$.
 The total number of gradient ascent steps required is $\tilde{\mathcal{O}} (\ell^{2}\rho^{0.5}\mu^{-2}\epsilon^{-1.5})$. Compared to the GRTR algorithm, the LMNegCur algorithm has the same complexity results.   
\end{rem}

%
\subsection{An inexact LMNegCur algorithm} 
		In this subsection, we propose an inexact LMNegCur (ILMNegCur) algorithm that allows approximate solutions to the minimum eigenvalue subproblem \eqref{alg2:3} and the linear system subproblem \eqref{alg2:2} of algorithm \ref{alg:2} under certain criteria.
		
	Specifically, to update $x$, the ILMNegCur algorithm applies the randomized Lanczos procedure \cite{Kuczynski1992Estimating} to $H_t$, thereby obtaining the approximate minimum eigenvalue $\hat{\lambda}_t$ and unit vector $u_t$ of $H_t$, such that
	\begin{align}\label{alg4:2}
		\hat{\lambda}_t \leqslant \lambda_{\min}(H_t)+\frac{1}{4}\sqrt{L_2\max\{\|g_t\|,\epsilon\}}, ~\hat{\lambda}_t =u_t^\top H_tu_t, ~g_t^{\top}u_t \leqslant0.
	\end{align}
	If $\hat{\lambda}_t \leqslant -\frac{1}{4}\sqrt{L_2\max\{\|g_t\|,\epsilon\}}$, then set
	\begin{align}\label{alg4:3}
		s_t= \frac{1}{2}\sqrt{\max\{\|g_t\|,\epsilon\}/L_2}u_t;
	\end{align}
	Otherwise, if $\|g_t\|\geqslant \epsilon$ and $\hat{\lambda}_t \geqslant -\dfrac{1}{4}\sqrt{L_2\|g_t\|}$, then use the conjugate gradient method \cite{Royer2018ComplexityAO} to solve the linear system subproblem \eqref{alg2:2} and get the direction $s_t$ such that
	\begin{align}\label{alg4:4}
		\|g_t+(H_t+ \sqrt{L_2\|g_t\|}I)s_t\| \leqslant \frac{1}{4}\min\bigg\{ \|g_t\|, \frac{1}{2}\sqrt{L_2\|g_t\|}\|s_t\|\bigg\}.
	\end{align}
	Note that if $\|g_t\|\leqslant \epsilon$ and $\hat{\lambda}_t \geqslant -\dfrac{1}{4}\sqrt{L_2\epsilon}$, we have found an $\mathcal{O}(\epsilon,\sqrt{\epsilon})$-second-order stationary point $x_t$.
	By replacing Steps 4 and 5 in the LMNegCur algorithm with the above steps, the detailed ILMNegCur algorithm can be immediately obtained.
	\subsection{Complexity analysis}
	
	\begin{lem}\label{lem4.1}
		Suppose that Assumption \ref{ass} holds. Let $\{x_t\}$ be a sequence generated by the ILMNegCur algorithm. Let
		\begin{equation}\label{lem4.1:1}
				\epsilon_1  =\min\bigg\{\frac{1}{32}, \frac{1}{144}, \frac{\sqrt{L_2}}{32L_1}\bigg\}\epsilon^{3/2}, ~\epsilon_2 = \min\bigg\{\frac{1}{12},\frac{1}{36}\bigg\}\sqrt{L_2}\epsilon^{1/2},
			\end{equation} 
		 and select $\eta_y, \theta, N_t$ according to \eqref{gradientascent}. In the $t$th iteration, if $\|g_t\|\geqslant \epsilon$ and $\hat{\lambda}_t \geqslant -\frac{1}{4}\sqrt{L_2\|g_t\|}$, then
		\begin{align*}
				&\|g_{t+1}\| \leqslant \frac{55}{8}\|g_t\|,  \\
				&P(x_{t+1})  \leqslant P(x_t)- \dfrac{1}{24}\sqrt{L_2\|g_t\|}\bigg(\frac{-3\sqrt{L_2\|g_t\|} + \sqrt{9L_2\|g_t\| + 8L_2\|g_{t+1}\|}}{2L_2}\bigg)^2;
		\end{align*}
	    If $\|g_{t+1}\| \geqslant \frac{1}{2}\|g_t\|$ holds, then we have
		\begin{align}\label{lem4.1:4}
				P(x_{t+1}) \leqslant P(x_t)-\frac{1}{288\sqrt{L_2}}\|g_t\|^{3/2}.
		\end{align}
	\end{lem}
	\begin{proof}
	By a proof similar to \eqref{lem5:4}, we have
		\begin{align}\label{lem4.1:5}
				&~P(x_{t+1}) - P(x_t)\nonumber\\
				\leqslant&~ g_t^\top s_t + \frac{1}{2} s_t^\top H_t s_t + \epsilon_1\|s_t\| 
				+ \frac{\epsilon_2}{2} \|s_t\|^2 + \frac{L_2}{6} \|s_t\|^3 \nonumber\\
				=&~ s_t^\top(g_t +  H_t s_t + \sqrt{L_2\|g_t\|}s_t) -\frac{1}{2} s_t^\top (H_t + \frac{1}{2}\sqrt{L_2\|g_t\|}I) s_t-\frac{3}{4}\sqrt{L_2\|g_t\|}\|s_t\|^2 \nonumber\\
				&~ + \epsilon_1\|s_t\| 
				+ \frac{\epsilon_2}{2} \|s_t\|^2 + \frac{L_2}{6} \|s_t\|^3.
		\end{align}
		Combining \eqref{alg4:2} with $\hat{\lambda}_t \geqslant -\dfrac{1}{4}\sqrt{L_2\|g_t\|}$, we can easily see that
		\begin{align}\label{lem4.1:6}
			\lambda_{\min}(H_t) \geqslant \hat{\lambda}_t  -\dfrac{1}{4}\sqrt{L_2\|g_t\|} \geqslant -\dfrac{1}{2}\sqrt{L_2\|g_t\|}.
		\end{align}
	    Then, by combining \eqref{lem4.1:5} with \eqref{alg4:4} and \eqref{lem4.1:6}, we obtain that
		\begin{align}\label{lem4.1:7}
				P(x_{t+1}) - P(x_t) 
				\leqslant  - \frac{5}{8}  \sqrt{L_2\|g_t\|}\|s_t\|^2  + \epsilon_1\|s_t\| + \frac{\epsilon_2}{2} \|s_t\|^2 + \frac{L_2}{6} \|s_t\|^3.
			\end{align} 
		To further bound the right-hand side of \eqref{lem4.1:7}, we first give some bounds on $\|s_t\|$.
		On the one hand, by combining \eqref{alg4:4} with \eqref{lem4.1:6}, we then obtain that
		\begin{align}\label{lem4.1:8}
		       \|s_t\| = &~ \|(H_t+\sqrt{L_2\|g_t\|}I)^{-1}(-g_t+g_t+(H_t+\sqrt{L_2\|g_t\|}I)s_t)\| \nonumber \\
				 \leqslant&~
				 \|(H_t+\sqrt{L_2\|g_t\|}I)^{-1}\|\big(\|g_t\|+\frac{1}{4}\|g_t\|\big)\leqslant \frac{5}{2}\sqrt{\dfrac{\|g_t\|}{L_2}}.
			\end{align}
	   On the other hand, by further combining \eqref{alg4:4} with  $\|H_t\| \leqslant \|H_t-\nabla^2 P(x_t)\|+ \|\nabla^2 P(x_t)\| \leqslant L_1+\epsilon_2$,  we obtain that
	   \begin{align*}
		   	    \|s_t\| \geqslant \frac{\|g_t\|-\|g_t+(H_t+ \sqrt{L_2\|g_t\|}I)s_t\|}{L_1 + \epsilon_2 + \sqrt{L_2\|g_t\|}} \geqslant   \frac{\frac{3}{4}\|g_t\|}{L_1 + \epsilon_2 + \sqrt{L_2\|g_t\|}},  
		   \end{align*}
		which implies that 
		\begin{align}\label{lem4.1:9}
				\frac{4}{3}\epsilon_2\|s_t\| + \frac{4}{3}(L_1 + \sqrt{L_2})\sqrt{\|g_t\|}\|s_t\|\geqslant \epsilon^{1/2}\|g_t\| \geqslant \epsilon^{3/2}.
			\end{align} 
		The rest of the proof is basically the same as Lemma \ref{lem3.1}. We omit it here for simplicity.
	\end{proof}
	
	\begin{lem}\label{lem4.2}
		Suppose that Assumption \ref{ass} holds. Let $\{x_t\}$ be a sequence generated by the ILMNegCur algorithm with parameters chosen the same as in Lemma \ref{lem4.1}. At the $t$th iteration, if  $\hat{\lambda}_t \leqslant -\frac{1}{4}\sqrt{L_2\max\{\|g_t\|,\epsilon\}}$, then
		\begin{align}
				&P(x_{t+1}) \leqslant P(x_t)-\frac{1}{288\sqrt{L_2}}\max\{\epsilon^{3/2},\|g_t\|^{3/2}\}, \label{lem4.2:2}\\ 
				&\|g_{t+1}\| \leqslant 1 + \frac{L_1}{2\sqrt{L_2}}\sqrt{\max\{\|g_t\|,\epsilon\}} + \max\{\|g_t\|,\epsilon\}. \label{lem4.2:3}
		\end{align}
	 If $\|g_t\|\leqslant \epsilon$ and $\hat{\lambda}_t \geqslant -\dfrac{1}{4}\sqrt{L_2\epsilon}$,  we have found an $\mathcal{O}(\epsilon,\sqrt{\epsilon})$-second-order stationary point $x_t$.
	\end{lem}
The proof of Lemma \ref{lem4.2} is basically the same as the proofs of Lemma \ref{lem3.2} and Lemma \ref{lem3.4}, and will not be repeated here.

	We are now ready to establish the iteration complexity for the ILMNegCur algorithm.
	In particular, let $\epsilon > 0$ be any given target accuracy. We  denote the first iteration index of the inexact LMNegCur algorithm to achieve $\|g_t\| \leqslant \epsilon $ and $\hat{\lambda}_t \geqslant -\frac{1}{4}\sqrt{L_2\epsilon}$ by 
	\begin{align*}
		\tilde{T}(\epsilon) := \min\left\{t |\|g_t\| \leqslant \epsilon ~\text{and}~ \hat{\lambda}_t \geqslant -\frac{1}{4}\sqrt{L_2\epsilon}\right\}.
	\end{align*}
	Then, by Lemma \ref{lem4.2}, we know that  $x_{\tilde{T}(\epsilon)}$ is already an $\mathcal{O}(\epsilon,\sqrt{\epsilon})$-second-order stationary point.
	To facilitate complexity analysis, we define the following indicator set:
	\begin{align}
		& \mathcal{F}_{\tilde{T}(\epsilon)} = \left\{t < \tilde{T}(\epsilon): P(x_{t+1})-P(x_t)\leqslant -\frac{1}{288\sqrt{L_2}} \max\{\|g_t\|^{3/2}, \epsilon^{3/2}\}\right\}, \\
		& \mathcal{G}_{\tilde{T}(\epsilon)} = \left\{t < \tilde{T}(\epsilon): \|g_{t+1}\|\leqslant \frac{1}{2}\|g_t\|, t \notin \mathcal{F}_{\tilde{T}(\epsilon)} \right\}.
	\end{align}
	By Lemma \ref{lem4.1} and \ref{lem4.2}, we can conclude that each iteration $t$ belongs to at least one of the above two sets $\mathcal{F}_{\tilde{T}(\epsilon)}$ and $\mathcal{G}_{\tilde{T}(\epsilon)}$.
	
	Similar to the proof in Section 2.3, we give an upper bound on $\tilde{T}(\epsilon)$ in the following theorem.
	
	\begin{thm}\label{thm4.1}
		Suppose that Assumption \ref{ass} holds. Let $\{x_t\}$ be a sequence generated by the ILMNegCur algorithm with all parameters chosen the same as in Lemma \ref{lem4.1}. We have
		\begin{align*}
				\tilde{T}(\epsilon)  &\leqslant  \tilde{C}_{1}\epsilon^{-3/2}\ln\epsilon^{-1} + \tilde{C}_{2}\epsilon^{-3/2}+1 =\tilde{\mathcal{O}}(\ell^{1.5}\rho^{0.5}\mu^{-1.5}\epsilon^{-1.5}), \\
				\sum_{t=0}^{\tilde{T}(\epsilon)}N_t &\leqslant 2\sqrt{\kappa} \ln \frac{\sqrt{\kappa+1}\|y_{-1}-y^*(x_0)\|}{A} +2\sqrt{\kappa} \tilde{T}(\epsilon) \ln\frac{\sqrt{\kappa+1}(A + \kappa r\sqrt{\tilde{G}}) }{A},			
			\end{align*}
		where 
		$\tilde{C}_{1}=288\sqrt{L_2}(P(x_0)-P^*)/\ln2$, $\tilde{C}_{2}=(\ln\tilde{G}+\ln2)\tilde{C}_{1}$,
		$\tilde{G}:= \max\big\{\frac{55}{8}\tilde{C}, 1 + \frac{L_1}{2\sqrt{L_2}}\hat{C}^{1/2} + \tilde{C}, 2 +\frac{L_1}{2\sqrt{L_2}}\big\}$ with $\tilde{C}:=\left (288\sqrt{L_2}\left (P(x_0)-P^*\right )\right)^{2/3}$, and $A$ is defined as in Lemma \ref{gradientascent}.
	\end{thm}

  Next, we give the Hessian-vector products required by the ILMNegCur algorithm. Note that, according to Lemmas 9 and 11 in \cite{Royer2018ComplexityAO}, in each outer iteration, both the randomized Lanczos procedure and the conjugate gradient (CG) procedure require at most $\min\big\{n,\tilde{\mathcal{O}}(L_1^{0.5}L_2^{-0.25}\epsilon^{-0.25})\big\}$ Hessian-vector products to obtain a direction $s_t$ that satisfies conditions \eqref{alg4:2} and \eqref{alg4:4} with high probability. We immediately obtain the following corollary.
       
   \begin{cor}\label{cor2}
    	Suppose that Assumption \ref{ass} holds, and the matrix inverse computation  $(\nabla_{y y}^{2}f(x_t,y_t))^{-1}$ in $H_t$ is approximated via $\tilde{\mathcal{O}}(\sqrt{\kappa})$ terms of Chebyshev polynomials as in \cite{Luo2022FindingSS}. Let $\{x_t\}$ be a sequence generated by the ILMNegCur algorithm  and all other parameters chosen the same as in Lemma \ref{lem4.1}. Then, $x_{\tilde{T}(\epsilon)}$ is an $\mathcal{O}(\epsilon,\sqrt{\epsilon})$-second-order stationary point with high probability guarantee. 
    	the total number of Hessian-vector products  required is at most $\tilde{\mathcal{O}}(\ell^{2.25}\rho^{0.25}\mu^{-1.75}\epsilon^{-1.75})$.
   \end{cor}

\section{Numerical Results}
In this section, we compare the numerical performance of the proposed GRTR (IGRTR) and LMNegCur (ILMNegCur) algorithms with the gradient descent (GDA) algorithm,  the MCN (IMCN) algorithm \cite{Luo2022FindingSS}, the MINIMAX-TR algorithm, and  the MINIMAX-TRACE algorithm \cite{Yao2024TwoTR} in solving a synthetic minimax optimization problem (a highly artificial piecewise function with smooth components connecting the segments) and a nonconvex-strongly concave minimax problem with sinusoidal perturbation and an adversarial deep learning problem. The first two experiments are implemented using Python 3.6 on a laptop with an Intel Core i7 2.8GHz and 8GB of RAM, while the last experiment is conducted on a 2.2GHz CPU, 80GB of RAM, and an Nvidia RTX 4090D GPU.

\subsection{Synthetic minimax optimization problem}
We first consider the following nonconvex-strongly concave minimax problem \cite{Yao2024TwoTR}:
\begin{equation}\label{numerical problem1}
	\min \limits_{x\in \mathbb{R}^n}\max \limits_{y\in \mathbb{R}^m}g(x) - \frac{1}{2} y^2,
\end{equation}
where
\begin{align*}
	g(x)= \begin{cases}
		g_{i,1}(x) & x_1, \cdots, x_{i-1} \in [2\tau, 6\tau], x_i \in [0,\tau], x_{i+1},\cdots, x_n\in [0,\tau], \\
		&1\leqslant i\leqslant n-1;\\
		g_{i,2}(x)   &x_1, \cdots, x_{i-1} \in [2\tau, 6\tau], x_i \in [\tau, 2\tau], x_{i+1},\cdots, x_n \in [0,\tau], \\
		&1\leqslant i\leqslant n-1;\\
		g_{n,1}(x)  &x_1, \cdots, x_{n-1} \in [2\tau, 6\tau], x_n \in [0,\tau]; \\
		g_{n,2}(x)  &x_1, \cdots, x_{n-1} \in [2\tau, 6\tau], x_n \in [\tau, 2\tau]; \\
		g_{n+1, 1}(x)  &x_1, \cdots, x_{n} \in [2\tau, 6\tau],
	\end{cases}
\end{align*}
with
\begin{align*}
	&g_{i,1}(x) = \sum_{j=1}^{i-1}L(x_j - 4\tau)^2 -\gamma x_i^2 + \sum_{j=i+1}^{n}Lx_j^2 - (i-1)\nu, 1\leqslant i\leqslant n-1, \\
	&g_{i,2}(x) = \sum_{j=1}^{i-1}L(x_j - 4\tau)^2 + h(x) + \sum_{j=i+2}^{n}Lx_j^2 - (i-1)\nu, 1\leqslant i \leqslant n-1, \\
	&g_{n,1}(x) = \sum_{j=1}^{n-1}L(x_j - 4\tau)^2 -\gamma x_n^2 - (n-1)\nu, \\
	&g_{n,2}(x) = \sum_{j=1}^{n-1}L(x_j - 4\tau)^2 + h(x) - (n-1)\nu, \\
	&g_{n+1,1}(x) = \sum_{j=1}^{n}L(x_j - 4\tau)^2 - n\nu,
\end{align*}
and
\begin{align*}
	h(x)=\begin{cases}
		h_1(x_i) + h_2(x_i)x_{i+1}^2 &x_1, \cdots, x_{i-1}\in [2\tau, 6\tau],\ x_i\in [\tau, 2\tau], \\&x_{i+1},\cdots, x_n \in [0,\tau], 1\leqslant i\leqslant n-1;  \\
		h_1(x_n) &x_1, \cdots, x_{n-1}\in [2\tau, 6\tau],\ x_n \in [\tau, 2\tau];\\
		0 &\text{otherwise},
	\end{cases}
\end{align*}
where	
\begin{align*}
	h_1(x) &= -\gamma x^2 + \frac{(-14L + 10\gamma)(x-\tau)^3}{3\tau} + \frac{(5L - 3\gamma)(x-\tau)^4}{2\tau^2},\\
	h_2(x) &= -\gamma - \frac{10(L+\gamma)(x-2\tau)^3}{\tau^3} - \frac{15(L+\gamma)(x-2\tau)^4}{\tau^4} - \frac{6(L+\gamma)(x-2\tau)^5}{\tau^5},
\end{align*}
and
\[
L > 0,\ \gamma >0,\ \tau = e,\ \nu = -h_1(2\tau) + 4L\tau^2.
\]
Note that $g(x)$ is only defined on the following domain:
\begin{equation*}
	D_0 = \bigcup_{i=1}^{n+1}\left\{x\in \mathbb{R}^n: 6\tau \geqslant x_1,\cdots,x_{i-1}\geqslant 2\tau, 2\tau\geqslant x_i\geqslant 0, \tau\geqslant x_{i+1},\cdots,x_n\geqslant 0\right\}.
\end{equation*}
By Lemma A.3 in \cite{Du2017GradientDC}, we know that there is only one local optimum, i.e., $(4\tau,...,4\tau)^{\top}$ and $d$ stationary points, i.e.,
\[
(0,\cdots,0)^{\top}, (4\tau, 0,\cdots, 0)^{\top}, \cdots, (4\tau,\cdots, 4\tau,0)^{\top}.
\]

In the experiment, we set the number of inner iterations $N_t$ to $1000$, the dimension of the variable $y$ to $5$, and the dimension of the variable $x$ to $\{10,20\}$. The initial point $x_0$ is set to $(10^{-3},\cdots,10^{-3})^{\top}$, close to the stationary point $(0,\cdots,0)^{\top}$, and $y_0$ is a random vector. For the problem \eqref{numerical problem1}, the parameter $L$ is $\{1, 1.5, 2\}$, and the parameter $\gamma$ is $1$. Figure \ref{fig1} shows the numerical performance of the nine tested algorithms for solving the problem \eqref{numerical problem1} under different parameter selections, where the horizontal axis represents the running time and the vertical axis represents the difference between $P(x)$ and the optimal value $P^*$.

Since there are $n$ saddle points in problem \eqref{numerical problem1}, and it is difficult for the GDA algorithm to escape from the saddle point when it is trapped in the saddle point, this also causes the curve of the GDA algorithm to present multiple horizontal line segments. As can be seen from Fig. \ref{fig1}, all second-order algorithms can effectively escape from the saddle point. The MINIMAX-TR algorithm converges slowly because the trust-region radius is set to a small value. The proposed GRTR algorithm and LMNegCur algorithm are both better than the state-of-the-art trust-region algorithm MINIMAX-TRACE and the cubic regularization algorithm MCN. The ILMNegCur algorithm performs better than other second-order algorithms because the approximate minimum eigenvalue of the Hessian matrix is achieved very early in the iterations of the Lanczos process, while the IGRTR and IMCN algorithms require more iterations and running time to escape from the saddle point. Numerical results show the effectiveness of the proposed GRTR (IGRTR) algorithm and LMNegCur (ILMNegCur) algorithm.

\begin{figure}[t]
	\centering 
	\includegraphics[scale=0.25]{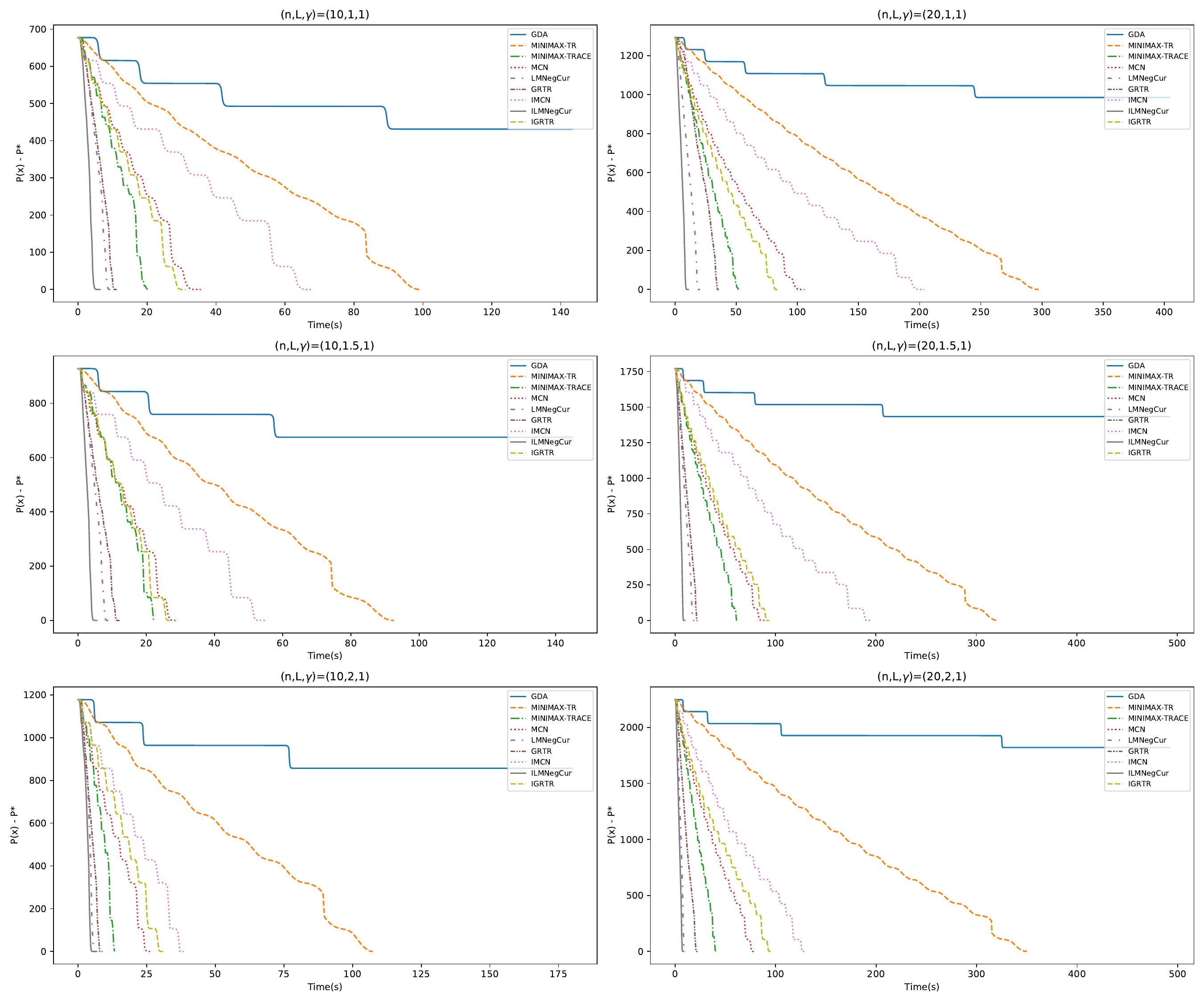}	
	\caption{Numerical results of the tested algorithms for solving \eqref{numerical problem1} with different choices of $n$, $L$ and $\gamma$.}
	\label{fig1}
\end{figure}

\subsection{Nonconvex-strongly concave problem with sinusoidal perturbation}
  We consider the following nonconvex-strongly concave problem with sinusoidal perturbation \cite{Zhang2024AGDA}:
\begin{equation}\label{numerical problem3}
	f(x,y)=\sin(\sqrt{L - 1}\cdot\sqrt{\|x\|^2 + 1})+\frac{1}{2}x^{\top}Qx+x^{\top}Ay-\frac{\mu_y}{2}\|y\|^2
\end{equation}
where \(A\in\mathbb{R}^{n\times m}\), \(Q\in\mathbb{R}^{n\times n}\), and \(\mu_y > 0\). 

In the experiment, we set \(m = n \in \{100,1000\}\), \(\mu_y = 1\) and \(L\in\{5, 20\}\), $A = V\Lambda_A V^{-1}$ and $Q = V\Lambda_Q V^{-1}$, where $V$ is an orthogonal matrix, $\Lambda_Q = \frac{\Lambda_Q^0}{\|\Lambda_Q^0\|_2}$, where $\Lambda_Q^0$ is a random diagonal matrix whose diagonal elements are uniformly randomly sampled from the interval $[-1,1]$, and \(\Lambda_A\) satisfies \(\Lambda_Q+\frac{1}{\mu_y^2}\Lambda_A^2\succeq0\). It is not difficult to verify that \(f(\cdot,\cdot)\) is a \(L\)-gradient Lipschitz continuous function. Figure \ref{fig3} shows the numerical performance of nine test algorithms for solving the problem \eqref{numerical problem3} under different parameter selections, where the horizontal axis represents the running time and the vertical axis represents the gradient norm of $P(x):=\max_{y} f(x,y)$. We run each algorithm until $\|\nabla P(x_t)\| \leqslant 1e-5$.

The results show that the GDA algorithm still converges relatively slowly.
Although the second-order methods have a higher cost per iteration, they still outperform the GDA algorithm in terms of running time, especially in converging to a high-quality solution. 

\begin{figure}[t]
	\centering 
	\includegraphics[scale=0.45]{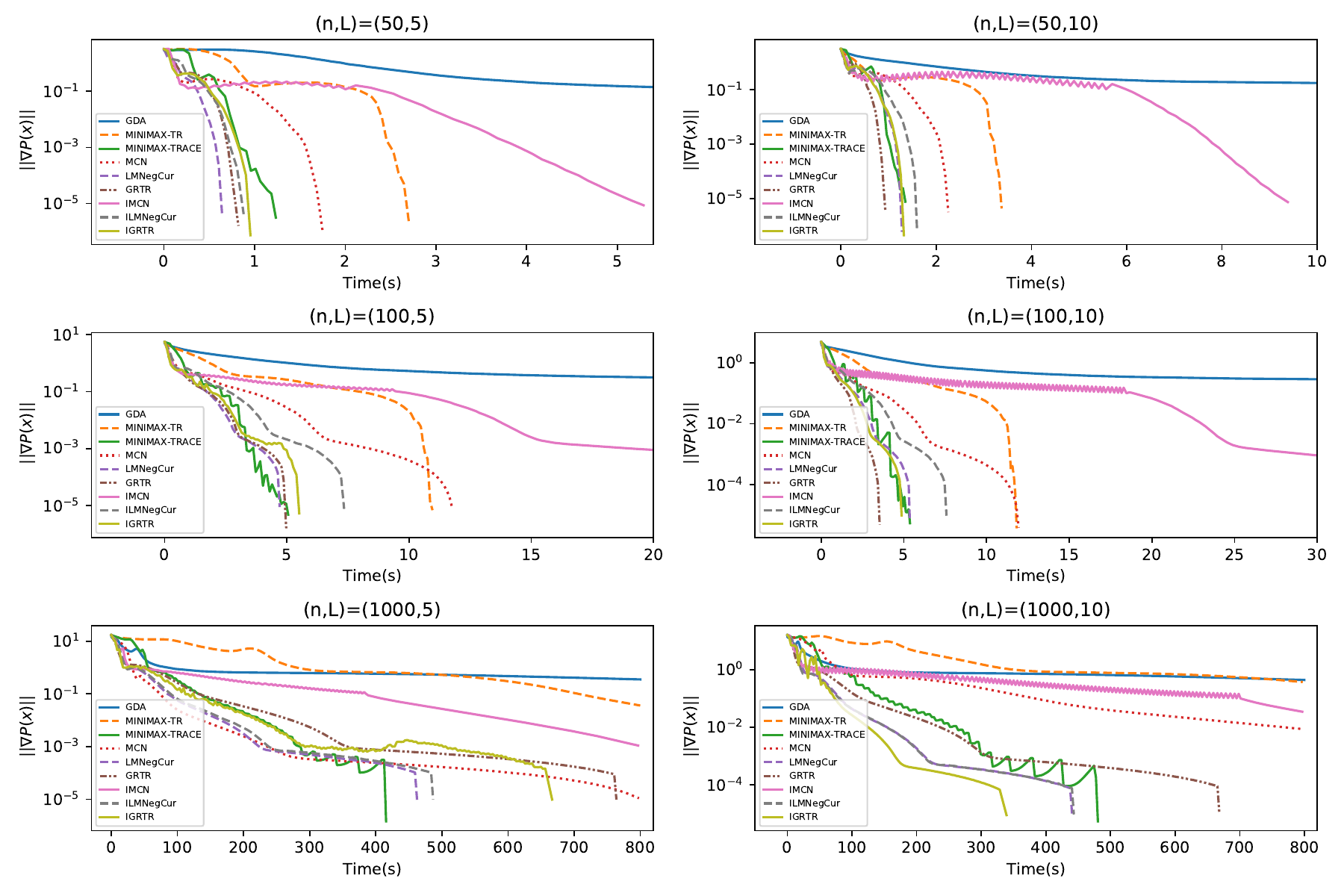}	
	\caption{Numerical results of the nine tested algorithms for solving \eqref{numerical problem3} with different choices of $n$ and $L$.}
	\label{fig3}
\end{figure}


\subsection{An adversarial deep learning problem}

The adversarial deep learning problem aims to train a convolutional neural network model with strong adversarially-robust capabilities, which can be formulated as the following minimax optimization problem \cite{Chen2023ACR}:
\begin{align}\label{numerical problem2}
	\min_{x} \max_{y=\{y_i\}_{i = 1}^{n}} \frac{1}{n} \sum_{i = 1}^{n} \left[ \ell(h_{x}(y_i), b_i)- \lambda \|y_i - a_i\|^2 \right], 
\end{align}
where $x$ is the parameter of the neural network $h_{x}(\cdot)$, $(a_i, b_i)$ corresponds to the \(i\)th image and label respectively, and $y_i$ represents the adversarial sample corresponding to $a_i$.

In the experimental setting, we choose the cross entropy loss function as \(\ell\) and the penalty coefficient as \(\lambda = 2\). We use the MNIST dataset \cite{Yann1998GradL}, which contains 50k training samples and 10k test samples. The network model we use is a convolutional neural network consisting of a convolutional block and a fully connected layer. Specifically, the convolutional block contains a convolutional layer and a sigmoid activation layer. The convolutional layer has $1$ input channel and $1$ output channel, a kernel size of $3$, a stride of $4$, and a padding of $1$. The input dimension of the connection layer is $49$ and the output dimension is $10$.

Since exact computation of the second-order subproblems can be a computationally expensive step for large-scale problems, we only compare the proposed IGRTR and ILMNegCur algorithms with the GDA, IMCN, and IMINIMAX-TR algorithms, which are inexact versions of the algorithms.
For all algorithms, we set the batch size to $64$. We select the gradient Lipschitz constant $L_1$ from $\{1, 5,10\}$ and the Hessian Lipschitz constant $L_2$ from $\{10,50,100\}$ for experiments. The experimental results showing the most stable and optimal performance are then shown. For the IGRTR algorithm, we set $N_t=30$, $\eta_y=0.01$, $\sigma=\sqrt{10}$, $r=1/\sqrt{10}$, and $\epsilon=0.01$, and solve the trust region subproblem in the same way as the trust region Newton-conjugate gradient method (Algorithm 4.1 in \cite{Curtis2021TrustN}). For the ILMNegCur algorithm, we set $N_t=30$, $\eta_y=0.01$, $L_2=50$, $\epsilon=0.01$, and solve the minimum eigenvalue subproblem by the Lanczos method, and the linear system subproblem by the conjugate gradient method. For the the IMINIMAX-TR algorithm, we set $N_t=30$, $\eta_y=0.01$, $L_2=50$, $\epsilon=0.01$, and solve the trust-region subproblem the same as in the IGRTR algorithm.
For the GDA algorithm, we set the step size of both $x$ and $y$ to 0.01. In addition, for the IMCN algorithm, we solve the subproblem about ${y}$  the same as the IGRTR algorithm and the ILMNegCur algorithm.
In addition, the cubic subproblem is solved by Algorithm 4 in \cite{Luo2022FindingSS}, where $L=2$, $M=50$, $\sigma=0.001$, and $\mathcal{K}(\epsilon,\delta^{'})=10$.

Figure \ref{fig2} shows the numerical performance of the five tested algorithms when solving problem \eqref{numerical problem2}, where in sub-figure (a), the horizontal axis represents the running time and the vertical axis represents the test accuracy; while in sub-figure (b), the horizontal axis represents the number of iterations $t$ and the vertical axis represents the function value $P(x)$.
It can be seen that the number of iterations and convergence time of all second-order algorithms are significantly less than that of the GDA algorithm, while IGRTR performs slightly better than other second-order algorithms. These results show the advantages of the two proposed algorithm.

\begin{figure}[htbp]
	\centering
	\subfigure[]{\includegraphics[height=4cm,width=6cm, trim=80pt 20pt 5pt 50pt]{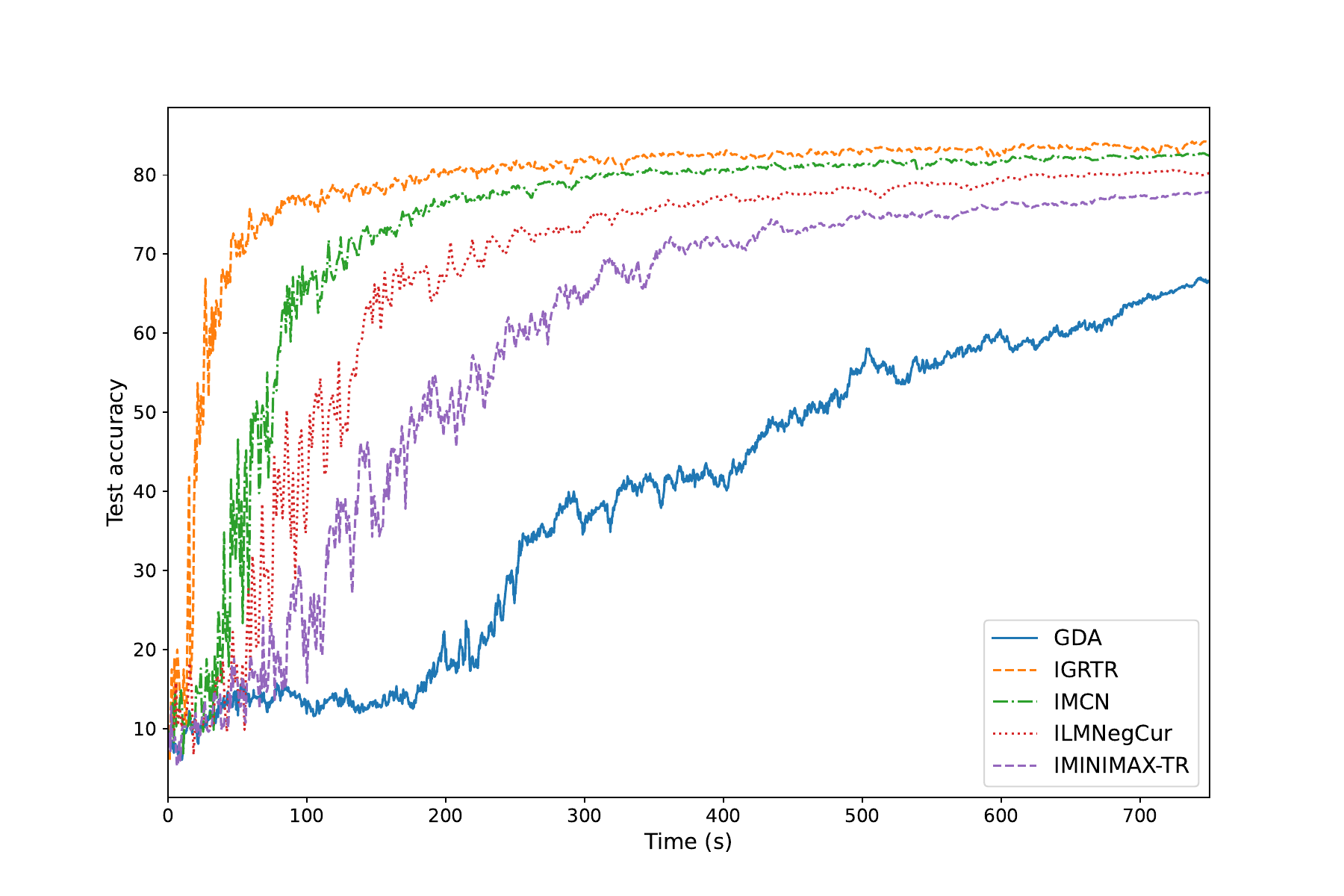}}
    \hfill 
	\subfigure[]{\includegraphics[height=4cm,width=6cm, trim=80pt 20pt 5pt 50pt]{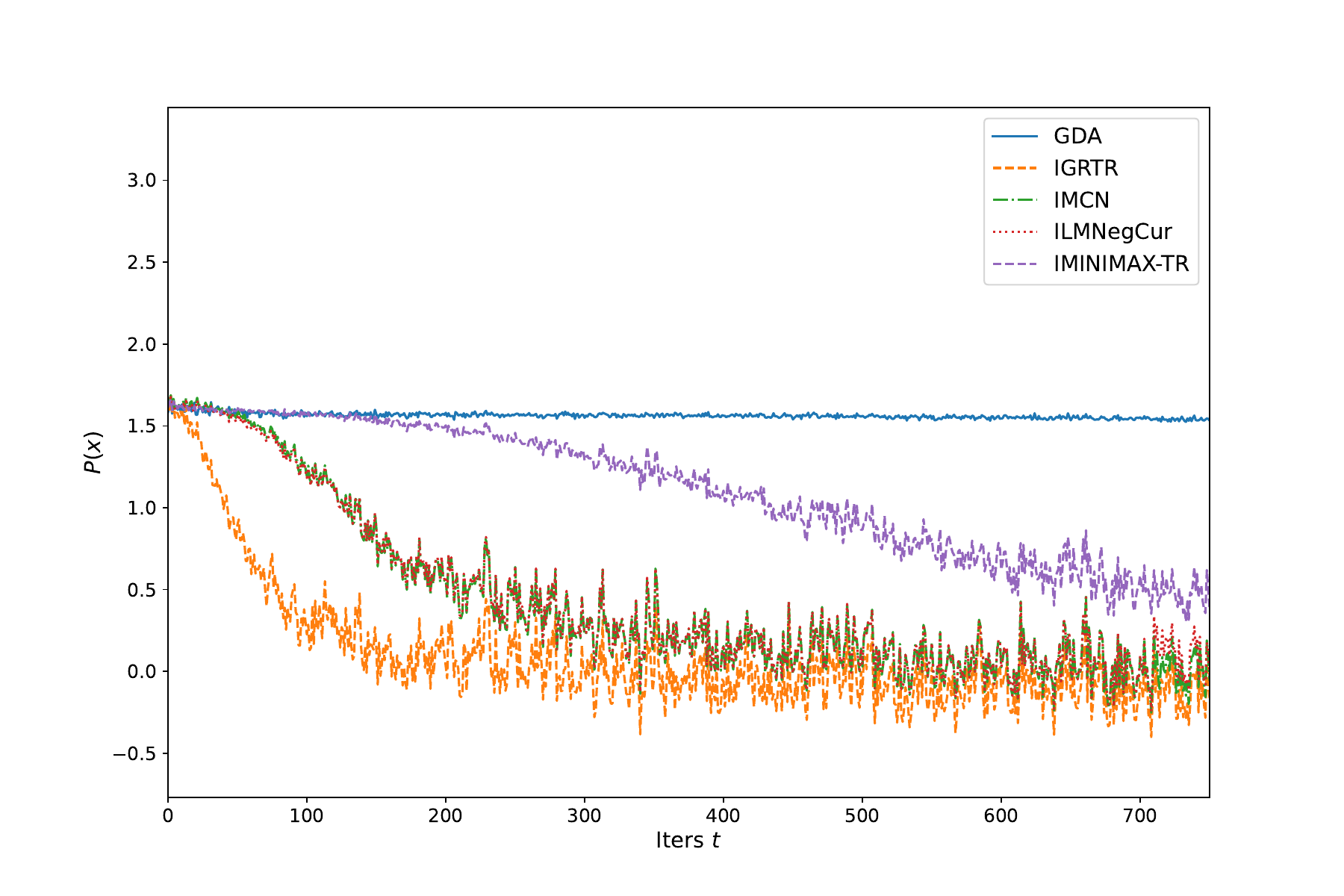}}
	\caption{ Numerical results of the five tested algorithms for solving \eqref{numerical problem2}.}
	\label{fig2}
\end{figure}

\section{Conclusions}
In this paper, we propose a gradient norm regularized trust-region (GRTR) algorithm and a Levenberg-Marquardt algorithm with a gradient norm regularization coefficient and use the negative curvature direction to correct the iteration direction (LMNegCur). Both the GRTR algorithm and the LMNegCur algorithm can achieve an $\mathcal{O}(\epsilon,\sqrt{\epsilon})$-second-order stationary point with an iteration complexity of $\tilde{\mathcal{O}}(\ell^{1.5}\rho^{0.5}\mu^{-1.5}\epsilon^{-1.5})$. The inexact variants of both methods require calling $\tilde{\mathcal{O}}(\ell^{2.25}\rho^{0.25}\mu^{-1.75}\epsilon^{-1.75})$ times of Hessian-vector products and $\tilde{\mathcal{O}}(\ell^{2}\rho^{0.5}\mu^{-2}\epsilon^{-1.5})$ times of gradient ascent steps to find $\mathcal{O}(\epsilon,\sqrt{\epsilon})$-second-order stationary points with high probability. It is worthwhile to further study the application of random subsampling to further improve the performance and complexity of the proposed algorithm for large-scale nonconvex minimax optimization problems.

\end{document}